\documentclass[10pt,oneside,leqno]{amsart}

\usepackage{amsmath,amsfonts,amssymb}
\usepackage{graphicx}
\usepackage[all]{xy}
\usepackage[english]{babel}
\usepackage[utf8x]{inputenc}
\usepackage{lscape}
\usepackage{pdflscape}
\usepackage{url}
\usepackage{booktabs}
\usepackage{multirow}
\usepackage{array}
\usepackage{ifthen}
\usepackage{placeins}

\newtheorem{thm}{Theorem}[section]
\newtheorem{prop}[thm]{Proposition}
\newtheorem{cor}[thm]{Corollary}

\theoremstyle{definition}

\newtheorem{rem}[thm]{Remark}
\newtheorem{ex}[thm]{Example}

\newcommand{\N}{\mathbb{N}}

\newcommand{\Q}{\mathbb{Q}}
\newcommand{\R}{\mathbb{R}}
\newcommand{\C}{\mathbb{C}}
\newcommand{\T}{\mathbb{T}}

\newcommand{\st}{\;:\;}

\newcommand{\sspace}{\cdot}
\newcommand{\ssspace}{\cdot\cdot}

\newcommand{\X}{X=\left.\Gamma\right\backslash G}
\newcommand{\g}{\mathfrak{g}}
\newcommand{\h}{\mathfrak{h}}

\newcommand{\duale}[1]{{#1}^*}

\newcommand{\kth}[1]{\ifthenelse{\equal{#1}{1}}{$#1^\text{st}$}{\ifthenelse{\equal{#1}{2}}{$#1^\text{nd}$}{\ifthenelse{\equal{#1}{3}}{$#1^\text{rd}$}{$#1^\text{th}$}}}}

\DeclareMathOperator{\im}{i}
\DeclareMathOperator{\imm}{im}
\DeclareMathOperator{\de}{d}

\newcommand{\del}{\partial}
\newcommand{\delbar}{\overline{\del}}

\renewcommand{\Re}{\mathsf{Re}}
\renewcommand{\Im}{\mathsf{Im}}

\title{Degree of non-K\"ahlerianity for $6$-dimensional nilmanifolds}

\author{Daniele Angella}
\address[Daniele Angella]{Dipartimento di Matematica\\
Universit\`{a} di Pisa \\
Largo Bruno Pontecorvo 5, 56127\\ 
Pisa, Italy}
\email{angella@mail.dm.unipi.it}

\curraddr{Centro di Ricerca Matematica ``Ennio de Giorgi''\\
Collegio Puteano, Scuola Normale Superiore\\
Piazza dei Cavalieri 3\\
56126 Pisa, Italy
}
\email{daniele.angella@sns.it}
\email{daniele.angella@gmail.com}

\author{Maria Giovanna Franzini}
\address[Maria Giovanna Franzini]{Dipartimento Di Matematica\\
Universit\`{a} di Parma \\
Parco Area delle Scienze 53/A, 43124 \\
Parma, Italy}
\email{mariagiovanna.franzini@studenti.unipr.it}

\author{Federico Alberto Rossi}
\address[Federico Alberto Rossi]{Dipartimento di Matematica e Applicazioni\\
Universit\`{a} di Milano Bicocca\\
Via Cozzi 53, 20125 \\
Milano, Italy} 
\email{f.rossi46@campus.unimib.it}

\thanks{This work was supported by GNSAGA of INdAM.
During the preparation of this work, the first author has been granted with a research fellowship by Istituto Nazionale di Alta Matematica INdAM, and has been supported by the Project PRIN ``Varietà reali e complesse: geometria, topologia e analisi armonica'', by the Project FIRB ``Geometria Differenziale e Teoria Geometrica delle Funzioni'', and by GNSAGA of INdAM.\\\\
\indent To appear in {\em manuscripta mathematica}.
The final publication is available at Springer via \url{http://dx.doi.org/10.1007/s00229-015-0734-x}.
}

\begin{document}

\begin{abstract}
 We use Bott-Chern cohomology to measure the non-K\"ahlerianity of $6$-dimensional nilmanifolds endowed with the invariant complex structures in M. Ceballos, A. Otal, L. Ugarte, and R. Villacampa's classification, \cite{ceballos-otal-ugarte-villacampa}. We investigate the existence of pluriclosed metric in connection with such a classification.
\end{abstract}

\maketitle

\section*{Introduction}

The \emph{Bott-Chern cohomology} and the \emph{Aeppli cohomology} provide further tools to investigate the complex geometry of non-K\"ahler manifolds: they are, in a sense, a bridge between the Dolbeault cohomology and the de Rham cohomology of a complex manifold $X$. Introduced by A. Aeppli in \cite{aeppli}, the cohomology groups
$$
H^{\bullet,\bullet}_{BC}(X) \;:=\; \frac{\ker\del\cap\ker\delbar}{\imm\del\delbar}
\qquad \text{ and } \qquad
H^{\bullet,\bullet}_{A}(X) \;:=\; \frac{\ker\del\delbar}{\imm\del+\imm\delbar}
$$
have raised much interest in several areas of Mathematics and Physics: see, e.g., B. Bigolin \cite{bigolin, bigolin-2}, J. Varouchas \cite{varouchas}, L. Alessandrini and G. Bassanelli \cite{alessandrini-bassanelli}, and, more recently, M. Schweitzer \cite{schweitzer} and R. Kooistra \cite{kooistra} (in the context of cohomology theories), J.-M. Bismut \cite{bismut-cras11, bismut} (in the context of Chern characters), and L.-S. Tseng and S.-T. Yau \cite{tseng-yau-3} (in the framework of generalized geometry and type II string theory).

The Bott-Chern cohomology turns out to be particularly interesting in the non-K\"ahler case. In fact, for compact K\"ahler manifolds, or, more in general, for compact complex manifolds satisfying the \emph{$\del\delbar$-Lemma} (that is, the very special cohomological property that every $\del$-closed $\delbar$-closed $\de$-exact form is $\del\delbar$-exact), the Bott-Chern cohomology coincides with the Dolbeault cohomology.

In \cite{angella-tomassini-3}, the first author and A. Tomassini proved that, on a compact complex manifold $X$, a Fr\"olicher-type inequality relates the dimension of $H^{\bullet,\bullet}_{BC}(X)$ and the dimension of $H^{\bullet}_{dR}(X;\C)$, namely, \cite[Theorem A]{angella-tomassini-3},
\begin{equation}\label{eq:frolicher-bc}
 \text{ for every }k\in\N \;, \qquad \sum_{p+q=k}\left(\dim_\C H^{p,q}_{BC}(X) + \dim_\C H^{p,q}_{A}(X)\right) \;\geq\; 2\,\dim_\C H^{k}_{dR}(X;\C) \;,
\end{equation}
and that the equality in \eqref{eq:frolicher-bc} holds if and only if $X$ satisfies the $\del\delbar$-Lemma, \cite[Theorem B]{angella-tomassini-3}.

\medskip

A special class of manifolds, which turn out to provide useful examples in studying geometric and cohomological properties of (almost-)complex manifolds, is given by the class of \emph{nilmanifolds}, namely, compact quotients $X=\left.\Gamma\right\backslash G$ of a connected simply-connected nilpotent Lie group $G$ by a discrete co-compact subgroup $\Gamma$.

By K. Nomizu's theorem \cite[Theorem 1]{nomizu}, the finite-dimensional vector space $\wedge^\bullet\duale{\mathfrak{g}}$ of linear forms on the dual of the Lie algebra $\mathfrak{g}$ naturally associated to $G$ yields a minimal model for the de Rham complex $\left(\wedge^\bullet X,\, \de\right)$ of $X$; this means, in particular, that the de Rham cohomology of $X$ can be computed as the cohomology of a finite-dimensional subcomplex of $\left(\wedge^\bullet X,\, \de\right)$.

On the other hand, nilmanifolds cannot be formal, \cite[Theorem 1]{hasegawa}, and hence, in particular, neither they satisfy the $\del\delbar$-Lemma nor they admit K\"ahler metrics.

\medskip

The $6$-dimensional nilmanifolds can be classified in terms of their Lie algebra, up to isomorphisms, in $34$ classes, according to V.~V. Morozov's classification, \cite{morozov}, see also \cite{magnin}. As regards their complex geometry, the works by S.~M. Salamon \cite{salamon}, L. Ugarte and R. Villacampa \cite{ugarte-villacampa}, A. Andrada, M.~L. Barberis, and I. Dotti \cite{andrada-barberis-dotti}, and M. Ceballos, A. Otal, L. Ugarte, and R. Villacampa \cite{ceballos-otal-ugarte-villacampa} give a complete classification, up to equivalence, of the invariant complex structures on $6$-dimensional nilmanifolds.

\medskip

Even if they admit no K\"ahler metric, nilmanifolds can be endowed with other special Hermitian metrics, whose associated $(1,1)$-form satisfies conditions weaker than the K\"ahler condition, see, e.g., \cite{michelsohn, bismut--math-ann-1989, ugarte, fino-parton-salamon, adela-luis-raquel}. A special class of Hermitian metrics that recently have arisen interest in several areas of Mathematics and Physics is the class of \emph{pluriclosed} (also called \emph{strong-K\"ahler with torsion}, shortly, \emph{\textsc{skt}}) metrics, namely, the Hermitian metrics whose associated $(1,1)$-form is $\del\delbar$-closed, see, e.g., \cite{bismut--math-ann-1989, fino-parton-salamon, rossi-tomassini, enrietti-fino-vezzoni, cavalcanti-skt}.

\medskip

In this note, we are concerned especially in studying cohomological properties of the $6$-dimensional nilmanifolds endowed with invariant complex structures; furthermore, we investigate cohomological properties of special classes of nilmanifolds admitting pluriclosed metrics. More precisely, we compute explicitly the dimensions of the Bott-Chern cohomology groups for each of the complex structures in M. Ceballos, A. Otal, L. Ugarte, and R. Villacampa's classification, \cite{ceballos-otal-ugarte-villacampa}: in view of the Fr\"olicher-type inequality \eqref{eq:frolicher-bc} and of K. Hasegawa's theorem \cite[Theorem 1]{hasegawa}, such dimensions, together with the Betti numbers, measure, in some sense, the non-K\"ahlerianity of nilmanifolds. We provide some sufficient conditions assuring the validity of the $\del\delbar$-Lemma on a compact complex manifold. Then, we study invariant pluriclosed structures on $6$-dimensional nilmanifolds, proving that only seven of the classes in the cohomological classification by means of the Bott-Chern cohomology contain invariant complex structures admitting invariant pluriclosed metrics, see Theorem \ref{thm:existence-skt}. On a $6$-dimensional nilmanifold endowed with an invariant complex structure, the existence of pluriclosed metrics and the existence of balanced metrics are complementary properties: indeed, by \cite[Proposition 1.4]{fino-parton-salamon}, or also \cite[Remark 1]{alexandrov-ivanov}, any Hermitian metric being both pluriclosed and balanced is in fact K\"ahler, and by \cite[Theorem 1.2]{fino-parton-salamon} the pluriclosed property is satisfied by either all invariant Hermitian metrics or by none; we refer to \cite{adela-luis-raquel} for the complementary study of balanced structures on $6$-dimensional nilmanifolds: there, A. Latorre, L. Ugarte, and R. Villacampa study the behaviour of the cohomology in relation to the existence of balanced Hermitian metrics on $6$-dimensional nilmanifolds endowed with invariant complex structures. In \cite[Theorem 14]{rossi-tomassini}, the third author and A. Tomassini provided a classification of $8$-dimensional nilmanifolds endowed with a 
invariant complex structure such that every invariant Hermitian metric is pluriclosed; such classification consists of two classes, the first of which contains manifolds of the type $M^6\times \T^2$, where $M^6$ is a $6$-dimensional nilmanifold admitting pluriclosed structures and $\T^2$ is the standard $2$-dimensional torus endowed with the standard complex structure. Hence, we investigate cohomological properties of manifolds of such type $M^6\times \T^2$.

\medskip

During the preparation of the final draft, we have been informed by L. Ugarte that A. Latorre, L. Ugarte, and R. Villacampa performed similar computations in \cite{adela-luis-raquel}, with the aim to further investigate the behaviour of the Bott-Chern cohomology under deformations of the complex structure, and the possible relation between cohomological properties and the existence of balanced or strongly-Gauduchon metrics; we refer to \cite{adela-luis-raquel} for further details.

\bigskip

\noindent{\sl Acknowledgments.} The authors would like to warmly thank Adriano Tomassini for his constant support and encouragement and for many interesting conversations and useful suggestions. We wish to warmly thank Luis Ugarte, Adela Latorre and Raquel Villacampa for pointing out the preprint \cite{adela-luis-raquel} and for their several suggestions and remarks, which highly improved this paper. Thanks are also due to Nicola Enrietti for helpful suggestions concerning the use of Maple for Complex Geometry. Many thanks to the anonymous Referee, whose suggestions improved the presentation of the paper.

\section{Preliminaries and notation}

\subsection{Bott-Chern cohomology of non-K\"ahler manifolds}

Let $X$ be a compact complex manifold. Other than Dolbeault cohomology, another important tool to study the complex geometry of $X$ is provided by the {\em Bott-Chern} and {\em Aeppli cohomologies}:
$$ H^{\bullet,\bullet}_{BC}(X) \;:=\; \frac{\ker\del\cap\ker\delbar}{\imm\del\delbar}\;,\qquad H^{\bullet,\bullet}_{A}(X) \;:=\; \frac{\ker\del\delbar}{\imm\del+\imm\delbar} \;.$$
Note that, while the Dolbeault cohomology groups are not symmetric, one has that, for every $p,\,q\in\N$, the complex conjugation induces the isomorphisms $H^{p,q}_{BC}(X)\simeq H^{q,p}_{BC}(X)$ and $H^{p,q}_{A}(X)\simeq H^{q,p}_{A}(X)$.

\medskip

One has the natural maps of (bi-)graded $\C$-vector spaces
$$
\xymatrix{
 & H^{\bullet,\bullet}_{BC}(X) \ar[dl]\ar[dr]\ar[d] & \\
H^{\bullet,\bullet}_{\delbar}(X) \ar[dr] & H^\bullet_{dR}(X;\C) \ar[d] & H^{\bullet,\bullet}_{\del}(X) \ar[dl] \\
 & H^{\bullet,\bullet}_{A}(X) & \\
}
$$
which are, in general, neither injective nor surjective. A compact complex manifold is said to satisfy the \emph{$\del\delbar$-Lemma} if every $\del$-closed $\delbar$-closed $\de$-exact form is $\del\delbar$-exact, namely, if the map $H^{\bullet,\bullet}_{BC}(X)\to H^{\bullet}_{dR}(X;\C)$ is injective: this turns out to be equivalent to all the above maps being isomorphisms, \cite[Lemma 5.15, Remark 5.16, 5.21]{deligne-griffiths-morgan-sullivan}. The compact complex manifolds admitting a K\"ahler metric, or, more in general, belonging to \emph{class $\mathcal{C}$ of Fujiki}, \cite{fujiki}, satisfy the $\del\delbar$-Lemma, \cite[Lemma 5.11, Corollary 5.23]{deligne-griffiths-morgan-sullivan}.

\medskip

Fixed a $J$-Hermitian metric $g$ on $X$, consider the \kth{4} order self-adjoint elliptic differential operators
$$ \tilde\Delta_{BC} \;:=\;
\left(\del\delbar\right)\left(\del\delbar\right)^*+\left(\del\delbar\right)^*\left(\del\delbar\right)+\left(\delbar^*\del\right)\left(\delbar^*\del\right)^*+\left(\delbar^*\del\right)^*\left(\delbar^*\del\right)+\delbar^*\delbar+\del^*\del $$
and
$$ \tilde\Delta_{A} \;:=\; \del\del^*+\delbar\delbar^*+\left(\del\delbar\right)^*\left(\del\delbar\right)+\left(\del\delbar\right)\left(\del\delbar\right)^*+\left(\delbar\del^*\right)^*\left(\delbar\del^*\right)+\left(\delbar\del^*\right)\left(\delbar\del^*\right)^* \;, $$
see \cite[Proposition 5]{kodaira-spencer-3}, see also \cite[\S2.b, \S2.c]{schweitzer}, \cite[\S5.1]{bigolin}; by means of them, one gets a Hodge theory for the Bott-Chern, respectively the Aeppli, cohomology: indeed, one has the following isomorphisms:
$$ H^{\bullet,\bullet}_{BC}(X) \;\simeq\; \ker\tilde\Delta_{BC} \;=\; \ker\del\cap\ker\delbar\cap\ker\left(\del\delbar\right)^* \;, $$
\cite[Théorème 2.2]{schweitzer}, and
$$ H^{\bullet,\bullet}_{A}(X) \;\simeq\; \ker\tilde\Delta_{A} \;=\; \ker\left(\del\delbar\right)\cap\ker\del^*\cap\ker\delbar^* \;,$$
\cite[\S2.c]{schweitzer}. In particular, one gets, \cite[Corollaire 2.3, \S2.c]{schweitzer}, that
$$ \dim_\C H^{\bullet,\bullet}_{BC}(X) \;<\; +\infty \qquad \text{ and } \qquad \dim_\C H^{\bullet,\bullet}_{A}(X) \;<\; +\infty \;, $$
and that, for every $p,q\in\N$,
$$ *\colon H^{p,q}_{BC}(X) \stackrel{\simeq}{\longrightarrow} H^{\dim_\C X-q,\dim_\C X-p}_{A}(X) \;.$$

\medskip

For any $k\in\N$, set
$$ \Delta^k(X) \;:=\; \sum_{p+q=k} \left(\dim_\C H^{p,q}_{BC}(X)+\dim_\C H^{p,q}_{A}(X)\right) - 2\,b_k \;,$$
where $b_k:=\dim_\R H_{dR}^k(X;\R)$ is the \kth{k} Betti number of $X$. Note that, if $X$ satisfies the $\del\delbar$-Lemma, then $\Delta^k=0$ for every $k\in\N$.\\
The first author and A. Tomassini proved that, \cite[Theorem A]{angella-tomassini-3},
$$ \text{for every }k\in\N\;, \qquad \Delta^k(X)\geq 0 \;,$$
and that, \cite[Theorem B]{angella-tomassini-3}, if $\Delta^k(X)=0$ for every $k\in\N$, then $X$ satisfies the $\del\delbar$-Lemma.

\medskip

Note that $\Delta^k(X)=\Delta^{\dim_\C X-k}(X)$ for every $k\in\N$; hence, on a $6$-dimensional manifold endowed with a complex structure, just $\Delta^1(X)$, $\Delta^2(X)$, and $\Delta^3(X)$ have to be computed.

\subsection{Invariant complex structures on $6$-dimensional nilmanifolds}\label{subsec:inv-cplx-str}

Let $\X$ be a nilmanifold, that is, a compact quotient of a connected simply-connected nilpotent Lie group $G$ by a discrete co-compact subgroup $\Gamma$ of $G$; denote the Lie algebra associated to $G$ by $\g$, and let $\g_\C:=\g\otimes_\R\C$.

Dealing with \emph{($G$-left-)invariant} objects on $X$, we mean objects on $X$ induced by objects on $G$ which are invariant under the action of $G$ on itself given by left-translations; equivalently, through left-translations, any invariant object on $X$ is uniquely determined by an object on spaces constructed starting with the Lie algebra $\g$.

\medskip

V.~V. Morozov classified in \cite{morozov} the $6$-dimensional nilpotent Lie algebras, up to isomorphism, in $34$ different classes, see also \cite{magnin}. In \cite{salamon}, S.~M. Salamon identified the $18$ classes of $6$-dimensional nilpotent Lie algebra admitting a linear integrable complex structure: up to equivalence\footnote{We recall that two linear integrable complex structures $J$ and $J'$ on a Lie algebra $\mathfrak{g}$ are said to be {\em equivalent} if there exists an automorphism $F\colon \mathfrak{g}\to\mathfrak{g}$ of the Lie algebra such that $J=F^{-1}\circ J'\circ F$.}, the linear integrable non-nilpotent complex structures\footnote{We recall that a linear integrable complex structure $J$ on a $2n$-dimensional Lie algebra $\mathfrak{g}$ is called \emph{nilpotent} if there is a basis $\left\{\omega^1,\ldots,\omega^n\right\}$ for $\left(\mathfrak{g}^{1,0}\right)^*$ with respect to which the structure equations are of the form $\de\omega^j = \sum_{h<k<j}A_{hk}^j\,\omega^h\wedge\omega^k+\sum_{h,k<j}B_{h k}^j\,\omega^h\wedge\bar\omega^k$ with $A_{hk}^j,\,B_{h k}^j\in\C$.} have been classified by L. Ugarte and R. Villacampa in \cite{ugarte-villacampa}, the linear integrable Abelian complex structures\footnote{We recall that a linear integrable complex structure $J$ on a Lie algebra $\mathfrak{g}$ is called \emph{Abelian} if $\left[Jx,\,Jy\right]=\left[x,\,y\right]$ for any $x,y\in\mathfrak{g}$.} have been classified by A. Andrada, M.~L. Barberis, and I. Dotti in \cite{andrada-barberis-dotti}, and lastly the linear integrable nilpotent non-Abelian complex structures have been classified by M. Ceballos, A. Otal, L. Ugarte, and R. Villacampa in \cite{ceballos-otal-ugarte-villacampa}.

We recall in Table \ref{table:classification} the classification in \cite{ceballos-otal-ugarte-villacampa}.
As a matter of notation, we identify a Lie algebra $\mathfrak{g}$ (or the associated connected simply-connected Lie group $G$, or the associated nilmanifold $X=\left.\Gamma\right\backslash G$, where $\Gamma$ is a discrete co-compact subgroup of $G$) by its structure equations, namely, writing
$$ \mathfrak{h}_5 \;:=\; \left(0^4,\, 13+42,\, 14+23\right) \;:=:\; \left(0,\, 0,\, 0,\, 0,\, 13+42,\, 14+23\right) \;, $$
we mean that there exists a basis $\left\{e_j\right\}_{j\in\{1,\ldots,6\}}$ of $\mathfrak{g}$ such that, with respect to the dual basis $\left\{e^j\right\}_{j\in\{1,\ldots,6\}}$ of $\duale{\g}$, the structure equations are
$$
\left\{
\begin{array}{l}
 \de e^1 \;=\; \de e^2 \;=\; \de e^3 \;=\; \de e^4 \;=\; 0 \\[5pt]
 \de e^5 \;=\; e^1\wedge e^3+e^4\wedge e^2 \\[5pt]
 \de e^6 \;=\; e^1\wedge e^4+e^2\wedge e^3
\end{array}
\right. \;,
$$
where usually we also shorten, e.g., $e^{AB}:=e^A\wedge e^B$. Similarly, we identify a complex structure on a nilmanifold by its structure equations in terms of a coframe of the holomorphic cotangent bundle, namely, writing
$$ J_2 \;:=\; \left(0,\, 0,\, \omega^{12}\right) \qquad \text{ on } \mathfrak{h}_5 \;, $$
we mean that there exists an invariant coframe $\left\{\omega^1,\, \omega^2,\, \omega^3\right\}$ of the $\mathcal{C}^\infty(X;\C)$-module $T^{1,0}X$, where $X = \left. \Gamma \right\backslash G$ is a nilmanifold with Lie algebra $\mathfrak{h}_5$, such that the structure equations read
$$
\left\{
\begin{array}{l}
 \de \omega^1 \;=\; \de \omega^2 \;=\; 0 \\[5pt]
 \de \omega^3 \;=\; \omega^1\wedge\omega^2
\end{array}
\right. \;,
$$
where usually we shorten, e.g., $\omega^{A\bar B}:=\omega^A\wedge\bar\omega^B$.

Other than the conditions listed in Table \ref{table:classification}, as in \cite{ceballos-otal-ugarte-villacampa} we assume that the parameters satisfy $\lambda\geq 0$, $c\geq 0$, $B\in\C$, and $D\in\C$.

\subsection{Cohomologies of nilmanifolds}

The space $\wedge^\bullet\duale{\g}$, equivalently, the space of invariant differential forms on $X$, yields a sub-complex $\left(\wedge^\bullet\duale{\g},\,\de\right)$ of the de Rham complex, $\left(\wedge^\bullet X,\,\de\right)$, where $\de\colon\wedge^\bullet\duale{\g}\to\wedge^{\bullet+1}\duale{\g}$ is induced by $\de\lfloor_{\wedge^1\duale{\g}}\colon \wedge^{1}\duale{\g}\ni\alpha\mapsto\de\alpha :=-\alpha\left(\left[\sspace,\,\ssspace\right]\right) \in\wedge^{2}\duale{\g}$. By K. Nomizu's theorem \cite[Theorem 1]{nomizu}, the map $\left(\wedge^\bullet\duale{\g},\,\de\right)\to\left(\wedge^\bullet X,\,\de\right)$ is a quasi-isomorphism, that is, the de Rham cohomology of $X$ can be computed using just invariant forms. (The same result holds true, more in general, for completely-solvable solvmanifolds, as proved by A. Hattori, \cite[Corollary 4.2]{hattori}; a counter-example for general solvmanifolds has been provided in \cite[Corollary 4.2, Remark 4.3]{debartolomeis-tomassini} by P. de 
Bartolomeis and A. Tomassini.)

\medskip

Analogously, an invariant complex structure $J$ on $X$ induces a bi-graduation on the space of invariant forms, equivalently $\wedge^{\bullet,\bullet}\duale{\g}_\C$, and a sub-double-complex $\left(\wedge^{\bullet,\bullet}\duale{\g}_\C,\,\del,\,\delbar\right)\hookrightarrow\left(\wedge^{\bullet,\bullet} X,\,\del,\,\delbar\right)$; this inclusion induces an injective morphism in cohomology, \cite[Lemma 9]{console-fino}, namely,
\begin{equation}\label{eq:inclusione-dolbeault}
\frac{\ker\left(\delbar\colon \wedge^{\bullet,\bullet}\duale{\g}_\C\to\wedge^{\bullet,\bullet+1}\duale{\g}_\C\right)}{\imm\left(\delbar\colon\wedge^{\bullet,\bullet-1}\duale{\g}_\C\to\wedge^{\bullet,\bullet}\duale{\g}_\C\right)} \;\hookrightarrow \; 
H^{\bullet,\bullet}_{\delbar}(X)
\;.
\end{equation}
Similarly, one has an injective morphism into Bott-Chern cohomology, \cite[Lemma 3.6]{angella},
\begin{equation}\label{eq:inclusione-bott-chern}
\frac{\ker\left(\de\colon \wedge^{\bullet,\bullet}\duale{\g}_\C\to\wedge^{\bullet+1,\bullet}\duale{\g}_\C\oplus\wedge^{\bullet,\bullet+1}\duale{\g}_\C\right)}{\imm\left(\del\delbar\colon\wedge^{\bullet-1,\bullet-1}\duale{\g}_\C\to\wedge^{\bullet,\bullet}\duale{\g}_\C\right)} \;\hookrightarrow \;
H^{\bullet,\bullet}_{BC}(X)
\;.
\end{equation}
By saying that the Dolbeault, respectively the Bott-Chern, cohomology of $X$ is \emph{($G$-left-)invariant}, we mean that the inclusion \eqref{eq:inclusione-dolbeault}, respectively \eqref{eq:inclusione-bott-chern}, is actually an isomorphism.
The invariance of the Bott-Chern cohomology can be seen as a consequence of the invariance of the Dolbeault cohomology, as proved by the first author, \cite[Theorem 3.7]{angella}.\\
It is not known whether, for any invariant complex structure on a nilmanifold, the Dolbeault cohomology is invariant, \cite[Conjecture 1]{rollenske-survey}, see also \cite[page 5406]{cordero-fernandez-gray-ugarte}, \cite[page 112]{console-fino}. However, this turns out to be true for a large class of invariant complex structures: more precisely, the following result holds.

\begin{thm}[{\cite[Theorem 1]{sakane}, \cite[Main Theorem]{cordero-fernandez-gray-ugarte}, \cite[Theorem 2, Remark 4]{console-fino}, \cite[Theorem 1.10]{rollenske}, \cite[Theorem 3.8]{angella}}]
 Let $\X$ be a nilmanifold endowed with an invariant complex structure $J$.
 The Dolbeault cohomology, and the Bott-Chern cohomology of $X$ are invariant, provided one of the following conditions holds:
\begin{itemize}
 \item $X$ is holomorphically parallelizable, \cite[Theorem 1]{sakane};
 \item $J$ is an \emph{Abelian} complex structure (i.e., $\left[Jx,\,Jy\right]=\left[x,\,y\right]$ for any $x,y\in\mathfrak{g}$), \cite[Remark 4]{console-fino};
 \item $J$ is a \emph{nilpotent} complex structure (i.e., there is an invariant coframe $\left\{\omega^1,\ldots,\omega^n\right\}$ for $\left(T^{1,0}X\right)^*$ with respect to which the structure equations of $X$ are of the form
$$ \de\omega^j\;=\; \sum_{h<k<j}A_{hk}^j\,\omega^h\wedge\omega^k+\sum_{h,k<j}B_{h k}^j\,\omega^h\wedge\bar\omega^k $$
with $\left\{A_{hk}^j,\,B_{h k}^j\right\}_{j,h,k}\subset\C$), \cite[Main Theorem]{cordero-fernandez-gray-ugarte};
 \item $J$ is a \emph{rational} complex structure (i.e., $J\left(\mathfrak{g}_\Q\right)\subseteq \mathfrak{g}_\Q$ where $\mathfrak{g}_\Q$ is the rational structure for $\mathfrak{g}$ induced by $\Gamma$ ---where a rational structure for $\mathfrak{g}$ is a $\Q$-vector space such that $\mathfrak{g}=\mathfrak{g}_\Q\otimes_\Q\R$), \cite[Theorem 2]{console-fino};
 \item $\mathfrak{g}$ admits a torus-bundle series compatible with $J$ and with the rational structure induced by $\Gamma$, \cite[Theorem 1.10]{rollenske}.
\end{itemize}
\end{thm}

\medskip

We recall that, since every $6$-dimensional nilpotent Lie algebra, except $\h_7$, admits a stable torus-bundle series compatible with any linear complex structure and any rational structure, \cite[Theorem B]{rollenske}, one gets the following result.

\begin{cor}[{\cite[Corollary 3.10]{rollenske-survey}}]
 For any $6$-dimensional nilmanifold endowed with an invariant complex structure and with Lie algebra non-isomorphic to $\mathfrak{h}_7=\left(0,\, 0,\, 0,\, 12,\, 13,\, 23\right)$, both the Dolbeault cohomology and the Bott-Chern cohomology are invariant.
\end{cor}

\medskip

Lastly, we recall that, if $X$ satisfies the $\del\delbar$-Lemma, then the differential graded algebra $\left(\wedge^{\bullet}X,\,\de\right)$ is formal, \cite[Main Theorem]{deligne-griffiths-morgan-sullivan}. By K. Hasegawa's theorem \cite[Theorem 1]{hasegawa}, no nilmanifold is formal except for the torus. In particular, the $\del\delbar$-Lemma does not hold for any complex structure on a non-torus nilmanifold.

\section{Bott-Chern cohomology of $6$-dimensional nilmanifolds}\label{sec:bc-coh}

In this section, we provide the results of the computation of the Bott-Chern cohomology for each of the complex structure in M. Ceballos, A. Otal, L. Ugarte, and R. Villacampa's classification, \cite{ceballos-otal-ugarte-villacampa}. In view of \cite[Corollary 3.10]{rollenske-survey}, and noting that the unique (up to equivalence) complex structure on the nilmanifold corresponding to $\mathfrak{h}_7$ is a rational complex structure, one is reduced to study invariant $\tilde\Delta_{BC}$-harmonic forms. Note that the Dolbeault, respectively Bott-Chern, cohomology being invariant is not an up-to-equivalence property: that is, the Bott-Chern numbers $h^{p,q}_{BC}:=\dim_\C H^{p,q}_{BC}(X)$ in Table \ref{table:bott-chern-6} provide a complete picture \emph{except for} nilmanifolds with Lie algebra isomorphic to $\mathfrak{h}_7$.
Such computations have been performed with the aid of a symbolic calculations software. In Figure \ref{fig:3d}, we provide a graphical visualization of the non-K\"ahlerianity of such $6$-dimensional nilmanifolds by means of the degrees $\Delta^1$, $\Delta^2$, and $\Delta^3$.

\begin{thm}
 Consider a $6$-dimensional nilpotent Lie algebra endowed with a linear integrable complex structure. Equivalently, consider a $6$-dimensional nilmanifold with Lie algebra non-isomorphic to $\mathfrak{h}_7=\left(0,\, 0,\, 0,\, 12,\, 13,\, 23\right)$ and endowed with any invariant complex structure, or a nilmanifold with Lie algebra isomorphic to $\mathfrak{h}_7$ and endowed with the invariant complex structure in M. Ceballos, A. Otal, L. Ugarte, and R. Villacampa's classification, \cite{ceballos-otal-ugarte-villacampa}, see Table \ref{table:classification}. Then the dimensions of the Bott-Chern cohomology groups are provided in Table \ref{table:bott-chern-6}.
\end{thm}

\begin{rem}
 During the preparation of the final version of this note, L. Ugarte informed us on a similar paper in preparation by A. Latorre, L. Ugarte, and R. Villacampa, \cite{adela-luis-raquel}: they computed the dimensions of the Bott-Chern cohomology groups for any complex structure on a $6$-dimensional nilpotent Lie algebra using the classification of \cite{ceballos-otal-ugarte-villacampa}, with the aim to study variations of the cohomology under deformations of the complex structure, and the behaviour of the cohomology in relation to the existence of balanced Hermitian metrics or strongly-Gauduchon metrics. Since the main purpose (other than some results and, possibly, the techniques of computation) of our paper is partly different from \cite{adela-luis-raquel}, we believe that no crucial overlapping could distract the reader.
\end{rem}

\begin{rem}\label{rem:stab-bc}
 As a consequence of the computations, we notice that the degrees $\Delta^1\in\{0,\, 2\}$ and $\Delta^3\in\{0,\, 4,\, 8,\, 12\}$ do not depend on the chosen linear integrable complex structure on a fixed $6$-dimensional nilpotent Lie algebra, while $\Delta^2\in\{0,\, \ldots,\, 9\}$ does depend on it; more precisely, the numbers $h^{p,q}_{BC}$ depend just on the Lie algebra whenever $p+q$ odd (compare also \cite[Remark 5.2]{angella}). Furthermore, on a $6$-dimensional Lie algebra endowed with a linear integrable complex structure, one has that the Abelianity of the Lie algebra is equivalent to the condition $\Delta^3=0$.
\end{rem}

\section{Detecting the validity of the $\del\delbar$-Lemma}

In this section, we prove some straightforward formulas concerning
$$ \left\{ \Delta^k \;:=:\; \Delta^k(X) \;:=\; \sum_{p+q=k} \left(\dim_\C H^{p,q}_{BC}(X)+\dim_\C H^{p,q}_{A}(X)\right) - 2\,\dim_\C H^{k}_{dR}\left(X;\C\right) \right\}_{k\in\N} \;,$$
in order to detect the validity of the $\del\delbar$-Lemma on a compact complex manifold $X$ just in terms of few numbers.

\medskip

Let $X$ be a compact complex manifold of complex dimension $n$.
As a matter of notation, denote, for $k\in\N$,
$$ b_k \;:=\; \dim_\C H^k_{dR}(X;\C) $$
and, for $\sharp\in\left\{\del,\delbar,BC,A\right\}$ and $\left(p,q\right)\in\N^2$ and $k\in\N$,
$$ h^{p,q}_{\sharp} \;:=\; \dim_\C H^{p,q}_{\sharp}\left(X\right) \qquad \text{ and } \qquad h^k_{\sharp} \;:=\; \sum_{p+q=k} \dim_\C H^{p,q}_{\sharp}\left(X\right) \;;$$
we recall that the following equalities hold: for any $\left(p,q\right)\in\N^2$,
$$ h^{p,q}_{BC} \;=\; h^{q,p}_{BC} \;=\; h^{n-p,n-q}_{A} \;=\; h^{n-q,n-p}_{A} \qquad \text{ and } \qquad h^{p,q}_{\delbar} \;=\; h^{q,p}_{\del} \;=\; h^{n-p,n-q}_{\delbar} \;=\; h^{n-q,n-p}_{\del} $$
and, for any $k\in\N$,
$$ b_k \;=\; b_{2n-k} \;, \qquad h^k_{BC} \;=\; h^{2n-k}_{A} \;, \qquad \text{ and } \qquad h^k_{\delbar} \;=\; h^k_{\del} \;=\; h^{2n-k}_{\delbar} \;=\; h^{2n-k}_{\del} \;.$$

Denote the topological Euler-Poincaré characteristic of $X$ by
$$ \chi_{\text{top}} \;:=\; \sum_{k=0}^{2n} \left(-1\right)^k\, b_k \;=\; \sum_{k=0}^{2n}\sum_{p+q=k} \left(-1\right)^{k}\, h^{p,q}_{\delbar} \;,$$
see, e.g., \cite[Theorem VI.11.7]{demailly-agbook}.

One has the equalities
\begin{eqnarray}\label{eq:euler}
b_n \;=\; \left(-1\right)^n\, \chi_{\text{top}} + 2\, \sum_{k=0}^{n-1} \left(-1\right)^{n-k-1}\, b_k \qquad \text{ and } \qquad h^n_{\delbar} \;=\; \left(-1\right)^n\, \chi_{\text{top}} + 2\, \sum_{k=0}^{n-1} \left(-1\right)^{n-k-1}\, h^k_{\delbar} \;.
\end{eqnarray}

\begin{prop}
 Let $X$ be a compact complex manifold. For any $k\in\N$, it holds that $\Delta^{2k+1} \equiv 0\mod 2$.
\end{prop}

\begin{proof}
 Since
 \begin{eqnarray*}
 \Delta^{2k+1} &=& h^{2k+1}_{BC}+h^{2k+1}_{A} - 2\, b_{2k+1} \\[5pt]
 &=& 2\, \left(\sum_{\substack{p+q=2k+1\\p<q}} \left( h^{p,q}_{BC} + h^{p,q}_{A} \right) - b_{2k+1}\right) \;\equiv\; 0\mod 2 \;,
 \end{eqnarray*}
 the proposition follows.
\end{proof}

\begin{prop}
 Let $X$ be a nilmanifold endowed with a complex structure, and suppose that the complex dimension $n$ of $X$ is odd. Then, $\Delta^{n} \equiv 0 \mod 4$.
\end{prop}

\begin{proof}
 Since the Euler-Poincaré characteristic $\chi_{\text{top}}$ of a nilmanifold vanishes, one has
 \begin{eqnarray*}
 \Delta^{n} &=& h^{n}_{BC} + h^{n}_{A} - 2\, b_n \;=\; 2\, h^n_{BC} - 2\, b_n \\[5pt]
 &=& 4\, \left( \sum_{\substack{p+q=n\\p<q}} h^{p,q}_{BC} - \sum_{k=0}^{n-1} \left(-1\right)^{n-k-1}\, b_k \right) \;\equiv\; 0 \mod 4 \;,
 \end{eqnarray*}
 from which the proposition follows.
\end{proof}

\begin{prop}
 Let $X$ be a compact complex manifold of complex dimension $n$. If either
 \begin{itemize}
  \item $\Delta^k=0$ for every $k\equiv n \mod 2$, and
  \item for every $k\in\N$, it holds that $\sum_{p+q=2k+1} \dim_\C \frac{\imm\delbar \cap \imm\del \cap \wedge^{p,q}X}{\imm\del\delbar} = 0$,
 \end{itemize}
 or
 \begin{itemize}
  \item $\Delta^k=0$ for every $k\equiv n-1 \mod 2$, and
  \item for every $k\in\N$, it holds that $\sum_{p+q=2k} \dim_\C \frac{\imm\delbar \cap \imm\del \cap \wedge^{p,q}X}{\imm\del\delbar} = 0$;
 \end{itemize}
 then $X$ satisfies the $\del\delbar$-Lemma.
\end{prop}

\begin{proof}
 Consider the bi-graded finite-dimensional $\C$-vector spaces
 $$ A^{\bullet,\bullet}\left(X\right) \;:=\; \frac{\imm\delbar\cap \imm\del}{\imm\del\delbar} \qquad \text{ and } \qquad F^{\bullet,\bullet}\left(X\right) \;:=\; \frac{\ker\del\delbar}{\ker\delbar+\ker\del} \;, $$
 and denote, for every $(p,q)\in\N^2$ and $k\in\N$,
 $$ a^{p,q} \;:=\; \dim_\C A^{p,q}\left(X\right) \qquad \text{ and } \qquad a^k \;:=\; \sum_{p+q=k} \dim_\C A^{p,q}\left(X\right) $$
 and
 $$ f^{p,q} \;:=\; \dim_\C F^{p,q}\left(X\right) \qquad \text{ and } \qquad f^k \;:=\; \sum_{p+q=k} \dim_\C F^{p,q}\left(X\right) \;.$$

 Following the same argument used in \cite{schweitzer} to prove the duality between Bott-Chern and Aeppli cohomology groups, we prove the duality between $A^{\bullet,\bullet}$ and $F^{\bullet,\bullet}$; it follows that, for any $\left(p,q\right)\in\N^2$ and $k\in\N$, the equalities
 $$ a^{p,q} \;=\; a^{q,p} \;=\; f^{n-p,n-q} \;=\; f^{n-q,n-p} \qquad \text{ and } \qquad a^k \;=\; f^{2n-k} $$
 hold.

 Indeed, note that the pairing
 $$ A^{\bullet,\bullet}\times F^{\bullet, \bullet} \to \C \;, \qquad \left(\left[\alpha\right],\, \left[\beta\right]\right) \mapsto \int_X \alpha\wedge \overline \beta \;, $$
 is non-degenerate: choose a Hermitian metric $g$ on $X$; if $\left[\alpha\right]\in A^{\bullet,\bullet}\subseteq H^{\bullet,\bullet}_{BC}(X)$, then there exists a $\tilde\Delta_{BC}$-harmonic representative $\tilde\alpha$ in $\left[\alpha\right]\in A^{\bullet,\bullet}$, by \cite[Corollaire 2.3]{schweitzer}, that is, $\del\tilde\alpha=\delbar\tilde\alpha=\del\delbar*\tilde\alpha=0$; hence, $\left[*\tilde\alpha\right]\in F^{\bullet,\bullet}$, and $\left(\left[\tilde\alpha\right],\, \left[*\tilde\alpha\right]\right)=\int_X \tilde\alpha\wedge \overline{*\tilde\alpha}$ is zero if and only if $\tilde\alpha$ is zero if and only if $\left[\alpha\right]\in A^{\bullet,\bullet}$ is zero.

 We further note that, by J. Varouchas' exact sequences, \cite[\S3.1]{varouchas}, one has the following equality, see \cite[\S3]{angella-tomassini-3}, for any $k\in\N$:
 \begin{eqnarray}\label{eq:bc-dolbeault}
  h^k_{BC} + h^k_{A} \;=\; 2\, h^k_{\delbar}+a^k+f^k \;.
 \end{eqnarray}

 By using \eqref{eq:bc-dolbeault} and \eqref{eq:euler}, one gets
 \begin{eqnarray*}
  \Delta^n &=& h^n_{BC}+h^n_{A} - 2\, b_n \;=\; 2\, h^n_{\delbar} + a^n + f^n - 2\, b_n \\[5pt]
           &=& 2\, \sum_{k=0}^{n-1} \left(-1\right)^{n-k-1}\, \left(2\, h^k_{\delbar}-2\, b_k\right) + a^n + f^n \\[5pt]
           &=& 2\, \sum_{k=0}^{n-1} \left(-1\right)^{n-k-1}\, \left(h^k_{BC}+h^k_{A}-2\, b_k\right) + 2\, \sum_{k=0}^{n-1} \left(-1\right)^{n-k}\, \left(a^k+f^k\right) + \left( a^n + f^n \right) \\[5pt]
           &=& 2\, \sum_{k=0}^{n-1} \left(-1\right)^{n-k-1}\, \Delta^k + 2\, \sum_{k=0}^{2n} \left(-1\right)^{n-k}\, a^k \;;
 \end{eqnarray*}
 in particular, it follows that
 $$ \sum_{k\equiv n \mod 2}\Delta^k + 2\, \sum_{k=1}^{n} a^{2k-1} \;=\; \sum_{k\equiv n-1 \mod 2} \Delta^k + 2\, \sum_{k=0}^{n} a^{2k} \;. $$

 Hence, the proposition follows by using \cite[Theorem B]{angella-tomassini-3}.
\end{proof}

In particular, by using the K. Hasegawa theorem \cite[Theorem 1, Corollary]{hasegawa}, one gets straightforwardly the following corollary.

\begin{cor}
 Let $X=\left.\Gamma\right\backslash G$ be a $6$-dimensional nilmanifold endowed with a complex structure. If either
 $$ \Delta^1 \;=\; \Delta^3 \;=\; 0 \qquad \text{ and } \qquad \dim_\C \frac{\imm\delbar \cap \imm\del \cap \wedge^{2,1}X}{\imm\del\delbar} \;=\; 0 \;, $$
 or
 $$ \Delta^2 \;=\; 0 \qquad \text{ and } \qquad \dim_\C \frac{\imm\delbar \cap \imm\del \cap \wedge^{1,1}X}{\imm\del\delbar} \;=\; 0 \;,$$
 then $X$ is diffeomorphic to a torus.
\end{cor}

\section{Pluriclosed metrics on $6$-dimensional nilmanifolds}

A very special class of metrics in Complex Geometry is provided by K\"ahler metrics, namely, by $J$-Hermitian metrics $g$ on a complex manifold $\left(X,\, J\right)$ such that the associated $(1,1)$-form $\omega:=g\left(J\,\sspace, \, \ssspace\right)$ is $\de$-closed. By Ch. Benson and C.~S. Gordon's theorem \cite[Theorem A]{benson-gordon-nilmanifolds}, or by K. Hasegawa's theorem \cite[Theorem 1, Corollary]{hasegawa}, if a nilmanifold admits a K\"ahler structure, then it is diffeomorphic to a torus. Hence, a natural question is to investigate the problem of the existence of special Hermitian metrics, other than the K\"ahler metrics, on nilmanifolds endowed with complex structures, and the possible connection with cohomological properties; for instance, given a complex manifold $\left(X,\, J\right)$ of complex dimension $n$, a Hermitian metric $g$ on $X$, with associated $(1,1)$-form $\omega$, is called \emph{balanced} if $\de\omega^{n-1}=0$, \cite{michelsohn}, and it is called \emph{pluriclosed} (or \emph{
strong K\"ahler with torsion}, shortly \emph{\textsc{skt}}) if $\del\delbar\omega=0$, \cite{bismut--math-ann-1989}. In particular, we study here pluriclosed metrics.

\medskip

The following result points out which of the complex manifolds in Table \ref{table:bott-chern-6} admit pluriclosed metrics.

\begin{thm}\label{thm:existence-skt}
 Let $X$ be a $6$-dimensional nilmanifold and $J$ be an invariant complex structure on $X$. Then $\left(X,\, J\right)$ admits an invariant pluriclosed metric if and only if
 \begin{enumerate}
  \item[{\bfseries [00]}]  $\left(X,\, J\right)=\left(\mathfrak{h}_1,\, J\right)$;\vspace{5pt}
  \item[{\bfseries [01b]}] $\left(X,\, J\right)=\left(\mathfrak{h}_2,\, J_1^D\right)$ with $D \in \{ \im, - \im \}$;\vspace{5pt}
  \item[{\bfseries [02b]}] $\left(X,\, J\right)=\left(\mathfrak{h}_2,\, J_2^D\right)$ with $D\in\C$ such that $\Re D=1$ and $\Im D>0$;\vspace{5pt}
  \item[{\bfseries [06b]}] $\left(X,\, J\right)=\left(\mathfrak{h}_4,\, J_2^D\right)$ with $D=1$;\vspace{5pt}
  \item[{\bfseries [09b']}] $\left(X,\, J\right)=\left(\mathfrak{h}_5,\, J_3^{\left(\lambda,\, D\right)}\right)$ with $\lambda=0$ and $D=\frac{1}{2}+\im y$ such that $0 < 4y^2 < 3$;\vspace{5pt}
  \item[{\bfseries [09c]}] $\left(X,\, J\right)=\left(\mathfrak{h}_5,\, J_3^{\left(\lambda,\, D\right)}\right)$ with $\lambda=0$ and $D=\frac{1}{2}$;\vspace{5pt}
  \item[{\bfseries [12]}]  $\left(X,\, J\right)=\left(\mathfrak{h}_8,\, J\right)$,
 \end{enumerate}
 where we use the notation in \S\ref{subsec:inv-cplx-str} and \S\ref{sec:bc-coh}, and in \cite{ceballos-otal-ugarte-villacampa}.
\end{thm}

\begin{proof}
It is well known that, if a $6$-dimensional nilmanifold endowed with an invariant complex structure admits an invariant pluriclosed metric, then its Lie algebra is isomorphic to $\mathfrak{h}_1$, $\mathfrak{h}_2$, $\mathfrak{h}_4$, $\mathfrak{h}_5$, or $\mathfrak{h}_8$, see \cite[Theorem 3.2]{fino-parton-salamon}, or \cite[Theorem 3.3]{ugarte}. Moreover, one knows that, on a $6$-dimensional nilmanifold endowed with an invariant complex structure, the pluriclosed condition is satisfied by either all the invariant Hermitian metrics or by none, \cite[Theorem 1.2]{fino-parton-salamon}.
Hence we consider only the complex structures on $\mathfrak{h}_1$, $\mathfrak{h}_2$, $\mathfrak{h}_4$, $\mathfrak{h}_5$, $\mathfrak{h}_8$ and test the pluriclosed condition on the standard invariant Hermitian metric $g$ with fundamental form $\sum_{j=1}^{3}\omega^j\wedge\bar\omega^j$, where $\left\{\omega^j\right\}_{j\in\{1,2,3\}}$ is a given invariant coframe for $T^{1,0}X$, according to the notation in \cite{ceballos-otal-ugarte-villacampa}.
Since the first case, $\mathfrak{h}_1$, is the K\"ahlerian one, we focus our attention on non-toric nilmanifolds, and, after a simple computation, we get the results summarized in Table \ref{table:skt-6}.
Imposing the pluriclosed condition on the standard metric and matching with the conditions in Table \ref{table:bott-chern-6}, we get the proposition.
\end{proof}

\begin{rem}
 As a consequence of the computations summarized in Table \ref{table:bott-chern-6} and of Theorem \ref{thm:existence-skt}, we note that, restricting to the class of $6$-dimensional nilmanifolds endowed with invariant pluriclosed structures, the number $h^{p,q}_{BC}$ depends just on the Lie algebra whenever $(p,q)\not\in\left\{(3,1),\, (2,2),\, (1,3)\right\}$, see also Remark \ref{rem:stab-bc}.
\end{rem}

With the aim of studying cohomological properties in relation to the existence of special metrics, we provide the following example, which shows a curve of complex structures admitting pluriclosed metrics and with jumping Bott-Chern cohomology numbers.

\begin{ex}[\itshape A curve of complex structures on a $6$-dimensional nilmanifold such that each complex structure admits a pluriclosed metric and the dimensions of the Bott-Chern cohomology groups are not constant]
Let $X$ be a nilmanifold with associated Lie algebra
$$ \mathfrak{h}_2 \;:=\; \left(0^4,\; 12,\; 34\right) \;. $$
Consider the family $\left\{J_t\right\}_{t\in\R}$ of invariant complex structures on $X$ such that an invariant coframe for the $\mathcal{C}^\infty\left(X\right)$-module of the $(1,0)$-forms on $\left(X,\, J_t\right)$ is given by $\left\{\psi_t^1,\, \psi_t^2,\, \psi_t^3\right\}$ satisfying
$$
\left\{
\begin{array}{rcl}
 \de \psi_t^1 &=& 0 \\[5pt]
 \de \psi_t^2 &=& 0 \\[5pt]
 \de \psi_t^3 &=& t\, \psi_t^1\wedge\psi_t^2 + \psi_t^1\wedge\bar\psi_t^1 + t\, \psi_t^1\wedge\bar\psi_t^2 + \left(t^2+\im\right)\, \psi_t^2\wedge\bar\psi_t^2
\end{array}
\right. \;.
$$
Note that
\begin{itemize}
 \item for $t=0$, the invariant complex structure $J_t$ is equivalent to the invariant Abelian complex structure $J_1^D$ with $D=\im$ (case \textbf{01b} in Table \ref{table:bott-chern-6});
 \item for $t\neq 0$, the invariant complex structure $J_t$ is equivalent to the invariant non-Abelian complex structure $J_2^D$ with $D=1+\frac{1}{t^2}\,\im$ (case \textbf{02b} in Table \ref{table:bott-chern-6}),
\end{itemize}
where we use the notation in \S\ref{subsec:inv-cplx-str} and \S\ref{sec:bc-coh}, and \cite{ceballos-otal-ugarte-villacampa}.

In particular, by Theorem \ref{thm:existence-skt}, $\left(X,\, J_t\right)$ admits an invariant pluriclosed metric for every $t\in\R$. Notwithstanding, note that
$$ \dim_\C H^{3,1}_{BC}\left(X,J_0\right) \;=\; \dim_\C H^{1,3}_{BC}\left(X,J_0\right) \;=\; 3 \;, $$
and
$$ \dim_\C H^{3,1}_{BC}\left(X,J_t\right) \;=\; \dim_\C H^{1,3}_{BC}\left(X,J_t\right) \;=\; 2 \qquad \text{ for } t\neq 0 \;, $$
while, for $\left(p,q\right)\not\in\left\{\left(3,1\right),\, \left(1,3\right)\right\}$ one has $\dim_\C H^{p,q}_{BC}\left(X,J_0\right) = \dim_\C H^{p,q}_{BC}\left(X,J_t\right)$ for every $t\in\R$.
\end{ex}
\begin{rem}
 An analogous example is provided by a nilmanifold $X$ with associated Lie algebra $\mathfrak{h}_5$ and endowed with 
the family $\left\{J_t := J_3^{\left(0,\, \frac{1}{2}+\im\, t\right)}\right\}_{t \geq 0 ,\, 4\,t^2 < 3}$ of invariant complex structures on $X$ such that, with respect to an invariant coframe $\left\{\psi_t^1,\, \psi_t^2,\, \psi_t^3\right\}$ of $T^{1,0}X$, one has
$$
\left\{
\begin{array}{rcl}
 \de \psi_t^1 &=& 0 \\[5pt]
 \de \psi_t^2 &=& 0 \\[5pt]
 \de \psi_t^3 &=& \psi_t^1\wedge\psi_t^2 + \psi_t^1\wedge\bar\psi_t^1 + \left(\frac12+\im t\right) \psi_t^2\wedge\bar\psi_t^2,
\end{array}
\right.
$$ 
(case \textbf{09} in Table \ref{table:classification}); in fact, there exists a pluriclosed metric for every $t$ such that $t\geq 0$ and $4\, t^2 < 3$, and we have the jumping of 
$$ \dim_\C H^{2,2}_{BC}\left(X,J_0\right)\;=\; 8\;\neq\;7\;=\;\dim_\C H^{2,2}_{BC}\left(X,J_t\right) \quad \text{ for } 0 < t < \frac{\sqrt{3}}{2} $$
(cases \textbf{09c} and \textbf{09b'} in Table \ref{table:bott-chern-6}).
\end{rem}

\begin{rem}
 Note that cases \textbf{09b'} and \textbf{09b''} provide examples of linear integrable non-equivalent complex structures on $\mathfrak{h}_5$ having the same dimensions of the Bott-Chern cohomology groups but admitting, respectively not admitting pluriclosed metrics.
\end{rem}

For the sake of completeness, we provide the following example, compare also \cite[Theorem 5.9]{ceballos-otal-ugarte-villacampa}: it shows a curve of complex structures on a $6$-dimensional nilmanifold admitting both pluriclosed and balanced metrics.

\begin{ex}[\itshape A curve of compact complex manifolds endowed with special metrics]
Let $X$ be a nilmanifold with associated Lie algebra
$$ \mathfrak{h}_4 \;:=\; \left(0^4,\; 12,\; 14+23\right) \;, $$
and consider on it the family $\left\{J_2^D\right\}_{D\in\R\setminus\{0\}}$ of invariant non-Abelian complex structures whose $(1,0)$-forms are generated by $\left\{\omega_D^1,\, \omega_D^2,\, \omega_D^3\right\}$ as a $\mathcal{C}^\infty(X;\C)$-module and such that
$$
\left\{
\begin{array}{rcl}
 \de \omega^1_D &:=& 0 \\[5pt]
 \de \omega^2_D &:=& 0 \\[5pt]
 \de \omega^3_D &:=& \omega^{12}_D + \omega^{1\bar1}_D + \omega^{1\bar2}_D + D\,\omega^{2\bar2}_D
\end{array}
\right. \;.
$$

By \cite[Theorem 3.2]{fino-parton-salamon} and \cite[Theorem 26]{ugarte}, the Lie algebra $\mathfrak{h}_4$ admits both pluriclosed and balanced invariant Hermitian structures.

Let
$$ \Omega_D^{r,s,t,u,v,z} \;:=\; \im\left(r^2\,\omega_D^{1\bar1}+s^2\,\omega_D^{2\bar2}+t^2\,\omega_D^{3\bar3}\right)+u\,\omega_D^{1\bar2}-\bar u\,\omega_D^{2\bar1}+v\,\omega_D^{2\bar3}-\bar v\,\omega_D^{3\bar2}+z\,\omega_D^{1\bar3}-\bar z\,\omega_D^{3\bar1} $$
the generic $J_2^D$-Hermitian structure, where
$$ r,\;s,\;t \;\in\; \R\setminus\{0\} \qquad \text{ and } u,\;v,\;z \;\in\; \C $$
satisfy
$$
\left\{
\begin{array}{l}
 r^2\,s^2 \;>\; \left|u\right|^2 \\[5pt]
 s^2\,t^2 \;>\; \left|v\right|^2 \\[5pt]
 r^2\,t^2 \;>\; \left|z\right|^2 \\[5pt]
 r^2\,s^2\,t^2+2\,\Re\left(\im\,\bar u\,\bar v\,z\right) \;>\; t^2\,\left|u\right|^2+r^2\,\left|v\right|^2+s^2\,\left|z\right|^2
\end{array}
\right. \;,
$$
see, e.g., \cite[pages 3-4]{ugarte-villacampa}.

By computing
$$ \del\delbar\, \Omega_D^{r,s,t,u,v,z} \;=\; \im \, t^2 \left(1-D\right) \, \omega_D^{12\bar1\bar2} \;, $$
we have that $\Omega_D^{r,s,t,u,v,z}$ is pluriclosed if and only if $D=1$, compare also \cite[Theorem 1.2]{fino-parton-salamon}.

Computing
$$ \de\left(\Omega_D^{r,s,t,u,v,z}\right)^2 \;=\; \frac{1}{2} \, \rho_D^{r,s,t,u,v,z} \, \omega^{123\bar1\bar2}_D - \frac{1}{2} \, \overline{\rho_D^{r,s,t,u,v,z}} \, \omega^{12\bar1\bar2\bar3}_D $$
where
$$ \rho_D^{r,s,t,u,v,z} \;:=\; D\,t^2\,r^2-D\,\left|v\right|^2+t^2\,s^2-\left|z\right|^2+\im\,t^2\,u+v\,\bar z \;, $$
we have that $\Omega_D^{r,s,t,u,v,z}$ is balanced if and only if $D$, $r$, $s$, $t$, $u$, $v$, $z$ are such that $\rho_D^{r,s,t,u,v,z}=0$.

Firstly, suppose that $0<D<\frac{1}{4}$. Let $s,t\in\R\setminus\{0\}$ be such that $s^2 > \left(D+s^2\right)^2$ and consider the $J_2^D$-Hermitian metric
$$ \Omega_D^{s,t} \;:=\; \Omega_D^{1,s,t,\im\, \left(D+s^2\right),0,0} \;=\; \im\left(\omega^{1\bar1}+s^2\,\omega^{2\bar2}+t^2\,\omega^{3\bar3}\right)+\im\, \left(D+s^2\right) \cdot \left(\omega^{1\bar2}+\omega^{2\bar1}\right) \;; $$
a straightforward computation shows that $\Omega_D^{s,t}$ is balanced, see also \cite[\S5]{ceballos-otal-ugarte-villacampa}.

In the case $D=\frac{1}{4}$, it turns out that $J_2^D$ admits no balanced metric, see also \cite[Theorem 5.9]{ceballos-otal-ugarte-villacampa}. Indeed, for the sake of completeness, we note that, if $\tilde\Omega$ were a balanced metric, then one should have
$$ 0 \;=\; \int_X \tilde\Omega^2 \wedge \de\left(\omega_{\frac{1}{4}}^{3} - \overline{\omega_{\frac{1}{4}}^{3}}\right) \;=\; \int_X \tilde\Omega^2 \wedge \left(\sqrt{2}\, \omega_{\frac{1}{4}}^1+\frac{\sqrt{2}}{2}\,\omega^2_{\frac{1}{4}}\right) \wedge \overline{\left(\sqrt{2}\, \omega_{\frac{1}{4}}^1+\frac{\sqrt{2}}{2}\,\omega^2_{\frac{1}{4}}\right)} \;>\; 0 \;, $$
which yields an absurd.

Summarizing, we have a curve such that:
\begin{itemize}
 \item for $0<D<\frac{1}{4}$, there exists a balanced metric on $\left(X,\, J_2^D\right)$;
 \item for $D=\frac{1}{4}$, the complex manifold $\left(X,\, J_2^{\frac{1}{4}}\right)$ admits no balanced metric and no pluriclosed metric;
 \item for $D=1$, there exists a pluriclosed metric on $\left(X,\, J_2^1\right)$.
\end{itemize}

Compare also \cite[Theorem 5.9]{ceballos-otal-ugarte-villacampa}, where the same example is studied in order to prove that the existence of balanced metrics and the existence of strongly-Gauduchon metrics are not closed under deformations of the complex structure.
\end{ex}

\begin{rem}
 For the study of the behaviour of the cohomology in relation to the existence of Hermitian balanced metrics or strongly-Gauduchon metrics, we refer to \cite{adela-luis-raquel} by A. Latorre, L. Ugarte, and R. Villacampa, where, for any balanced structure, they also determine L.-S. Tseng and S.-T. Yau's spaces parametrizing the deformations in type IIB supergravity, see \cite{tseng-yau-3}.
\end{rem}

\section{$8$-dimensional nilmanifolds and pluriclosed metrics}

While, for $6$-dimensional nilmanifolds endowed with an invariant complex structure, the condition that a given invariant metric is pluriclosed depends just on the associated Lie algebra and not on the chosen invariant Hermitian metric, \cite[Theorem 1.2]{fino-parton-salamon}, this holds no more true in dimension greater than $6$, see \cite[\S2]{rossi-tomassini}. The $8$-dimensional nimanifolds endowed with an invariant complex structure such that every invariant Hermitian metric is pluriclosed have been classified by the third author and A. Tomassini in \cite[Theorem 13]{rossi-tomassini} into two families, compare also \cite[Theorem 4.1]{enrietti-fino-vezzoni} (as a consequence, it turns out that an invariant complex structure on a $8$-dimensional nilmanifold is pluriclosed for every invariant Hermitian metric if and only if it is astheno-K\"ahler\footnote{We recall that a Hermitian metric $g$ on a complex manifold $(X,J)$ of complex dimension $n$ is said to be {\em astheno-K\"ahler} (according to the terminology by J. Jost and S.-T. Yau) if the associated $(1,1)$-form $\omega:=g(J\sspace,\ssspace)$ satisfies $\del\delbar\omega^{n-2}=0$.} for every invariant Hermitian metric, \cite[Theorem 15]{rossi-tomassini}); more precisely, the first family consists of nilmanifolds of the type $M^6\times \T^2$, where $M^6$ is a $6$-dimensional nilmanifold endowed with an invariant complex structure such that any invariant Hermitian metric on it is pluriclosed, \cite[Theorem 1.2]{fino-parton-salamon}, and $\T^2$ is the standard $2$-dimensional torus endowed with the standard complex structure.
In this section, we study the Bott-Chern cohomology of such $8$-dimensional nilmanifolds $M^6\times \T^2$. Recall that the possible $M^6$ have been classified in \cite[Theorem 3.2]{fino-parton-salamon} by A. Fino, M. Parton, and S. Salamon in terms of their associated Lie algebra.

\medskip

More precisely, in Table \ref{table:skt-8dim} we summarize the nilmanifolds of such type $M^6\times \T^2$, that is, the nilmanifolds possibly endowed with invariant pluriclosed structures and belonging to the first class in \cite[Theorem 13]{rossi-tomassini}. Then, in Table \ref{table:bc-skt-8dim}, we summarize the results of the computations of the dimensions of the Bott-Chern cohomology groups and the degree of non-K\"ahlerianity for such $8$-dimensional nilmanifolds, pointing out the existence of pluriclosed structures. (As regards notation, we follow the same conventions as in \S\ref{subsec:inv-cplx-str} and \S\ref{sec:bc-coh}, and \cite{ceballos-otal-ugarte-villacampa}.)

\FloatBarrier

\appendix

\section*{Tables}

\begin{table}
\centering
{\resizebox{\textwidth}{!}{
\begin{tabular}{>{\textbf\bgroup}l<{\textbf\egroup} >{$}l<{$} | >{$}l<{$} >{$}l<{$} >{$}l<{$} | >{$}l<{$} >{$}l<{$}}
 \toprule
 $\mathbf{\sharp}$ & \text{\bfseries algebra} & b^1 & b^2 & b^3 & \text{\bfseries complex structure} & \text{\bfseries conditions} \left( \lambda \geq 0 ,\; c \geq 0 ,\; B \in \C ,\; D \in \C \right) \\
 & & & & & & \left( \mathcal{S}(B,c) := c^4-2\,\left(\left|B\right|^2+1\right)\, c^2 + \left(\left|B\right|^2-1\right)^2 \right) \\
 \toprule
00 & \mathfrak{h}_1\;:=\;\left(0,\; 0,\; 0,\; 0,\; 0,\; 0\right) & 6 & 15 & 20 & J\;:=\;\left(0,\; 0,\; 0\right) & \\[3pt]
 \midrule[0.02em]
01 & \mathfrak{h}_2\;:=\;\left(0,\; 0,\; 0,\; 0,\; 12,\; 34\right) & 4 & 8 & 10 & J_1^D\;:=\;\left(0,\; 0,\; \omega^{1\bar1}+D\,\omega^{2\bar2}\right) \;, & \Im D\in\{1,-1\} \\[3pt]
02 & & & & & J_2^D\;:=\;\left(0,\; 0,\; \omega^{12}+\omega^{1\bar1}+\omega^{1\bar2}+D\,\omega^{2\bar2}\right) \;, & \Im D>0 \\[3pt]
 \midrule[0.02em]
03 & \mathfrak{h}_3\;:=\;\left(0,\; 0,\; 0,\; 0,\; 0,\; 12+34\right) & 5 & 9 & 10 &  J_1\;:=\;\left(0,\; 0,\; \omega^{1\bar1}+\omega^{2\bar2}\right) & \\[3pt]
04 & & & & &  J_2\;:=\;\left(0,\; 0,\; \omega^{1\bar1}-\omega^{2\bar2}\right) & \\[3pt]
 \midrule[0.02em]
05 & \mathfrak{h}_4\;:=\;\left(0,\; 0,\; 0,\; 0,\; 12,\; 14+23\right) & 4 & 8 & 10 &  J_1\;:=\;\left(0,\; 0,\; \omega^{1\bar1}+\omega^{1\bar2}+\frac{1}{4}\,\omega^{2\bar2}\right) & \\[3pt]
06 & & & & &  J_2^D\;:=\;\left(0,\; 0,\; \omega^{12}+\omega^{1\bar1}+\omega^{1\bar2}+D\,\omega^{2\bar2}\right) \;, & D\in\R\setminus\{0\}\\[3pt]
 \midrule[0.02em]
07 & \mathfrak{h}_5\;:=\;\left(0,\; 0,\; 0,\; 0,\; 13+42,\; 14+23\right) & 4 & 8 & 10 &  J_1^D\;:=\;\left(0,\; 0,\; \omega^{1\bar1}+\omega^{1\bar2}+D\,\omega^{2\bar2}\right) \;, & D\in\left[0,\,\frac{1}{4}\right) \\[3pt]
08 & & & & & J_2\;:=\;\left(0,\; 0,\; \omega^{12}\right) & \\[3pt]
09 & & & & & J_3^{\left(\lambda,\,D\right)}\;:=\;\left(0,\; 0,\; \omega^{12}+\omega^{1\bar1}+\lambda\,\omega^{1\bar2}+D\,\omega^{2\bar2}\right) \;, & \left(\lambda,\,D\right)\in \left\{\left(0,\, x+\im y\right)\in\R\times \C \st y\geq 0\;, 4y^2<1+4x \right\} \\[3pt]
 & & & & & & \cup \left\{\left(\lambda,\, \im y\right)\in\R\times \C \st 0<\lambda^2<\frac{1}{2} ,\; 0\leq y<\frac{\lambda^2}{2} \right\} \\[3pt]
 & & & & & & \cup \left\{\left(\lambda,\, \im y\right)\in\R\times \C \st \frac{1}{2}\leq\lambda^2<1 ,\; 0\leq y<\frac{1-\lambda^2}{2} \right\} \\[3pt]
 & & & & & & \cup \left\{\left(\lambda,\, \im y\right)\in\R\times \C \st \lambda^2>1 ,\; 0\leq y<\frac{\lambda^2-1}{2} \right\} \\[3pt]
 \midrule[0.02em]
10 & \mathfrak{h}_6\;:=\;\left(0,\; 0,\; 0,\; 0,\; 12,\; 13\right) & 4 & 9 & 12 & J\;:=\;\left(0,\; 0,\; \omega^{12}+\omega^{1\bar1}+\omega^{1\bar2}\right) & \\[3pt]
 \midrule[0.02em]
11 & \mathfrak{h}_7\;:=\;\left(0,\; 0,\; 0,\; 12,\; 13,\; 23\right) & 3 & 8 & 12 & J\;:=\;\left(0,\; \omega^{1\bar1},\; \omega^{12}+\omega^{1\bar2}\right) & \\[3pt]
 \midrule[0.02em]
12 & \mathfrak{h}_8\;:=\;\left(0,\; 0,\; 0,\; 0,\; 0,\; 12\right) & 5 & 11 & 14 & J\;:=\;\left(0,\; 0,\; \omega^{1\bar1}\right) & \\[3pt]
 \midrule[0.02em]
13 & \mathfrak{h}_9\;:=\;\left(0,\; 0,\; 0,\; 0,\; 12,\; 14+25\right) & 4 & 7 & 8 & J\;:=\;\left(0,\; \omega^{1\bar1},\; \omega^{1\bar2}+\omega^{2\bar1}\right) & \\[3pt]
 \midrule[0.02em]
14 & \mathfrak{h}_{10}\;:=\;\left(0,\; 0,\; 0,\; 12,\; 13,\; 14\right) & 3 & 6 & 8 & J\;:=\;\left(0,\; \omega^{1\bar1},\; \omega^{12}+\omega^{2\bar1}\right) & \\[3pt]
 \midrule[0.02em]
15 & \mathfrak{h}_{11}\;:=\;\left(0,\; 0,\; 0,\; 12,\; 13,\; 14+23\right) & 3 & 6 & 8 & J^B\;:=\;\left(0,\; \omega^{1\bar1},\; \omega^{12}+B\, \omega^{1\bar2}+\left|B-1\right|\,\omega^{2\bar1}\right) \;, & B\in\R\setminus\{0,\,1\} \\[3pt]
 \midrule[0.02em]
16 & \mathfrak{h}_{12}\;:=\;\left(0,\; 0,\; 0,\; 12,\; 13,\; 24\right) & 3 & 6 & 8 & J^B\;:=\;\left(0,\; \omega^{1\bar1},\; \omega^{12}+B\,\omega^{1\bar2}+\left|B-1\right|\,\omega^{2\bar1}\right) \;, & \Im B\neq0 \\[3pt]
 \midrule[0.02em]
17 & \mathfrak{h}_{13}\;:=\;\left(0,\; 0,\; 0,\; 12,\; 13+14,\; 24\right) & 3 & 5 & 6 & J^{(B,\,c)}\;:=\;\left(0,\; \omega^{1\bar1},\; \omega^{12}+B\, \omega^{1\bar2}+c\, \omega^{2\bar1}\right) \;, & c\neq\left|B-1\right|,\, \left(c,\, \left|B\right|\right)\neq\left(0,\,1\right),\,  \mathcal{S}(B,c) < 0\\[3pt]
 \midrule[0.02em]
18 & \mathfrak{h}_{14}\;:=\;\left(0,\; 0,\; 0,\; 12,\; 14,\; 13+42\right) & 3 & 5 & 6 & J^{(B,\,c)}\;:=\;\left(0,\; \omega^{1\bar1},\; \omega^{12}+B\, \omega^{1\bar2}+c\, \omega^{2\bar1}\right) \;, & c\neq\left|B-1\right|,\, \left(c,\, \left|B\right|\right)\neq\left(0,\,1\right),\, \mathcal{S}(B,c) = 0\\[3pt]
 \midrule[0.02em]
19 & \mathfrak{h}_{15}\;:=\;\left(0,\; 0,\; 0,\; 12,\; 13+42,\; 14+23\right) & 3 & 5 & 6 & J_1\;:=\;\left(0,\; \omega^{1\bar1},\; \omega^{2\bar1}\right) & \\[3pt]
20 & & & & & J_2^{c}\;:=\;\left(0,\; \omega^{1\bar1},\; \omega^{1\bar2}+c\, \omega^{2\bar1}\right) \;, & c \neq 1 \\[3pt]
21 & & & & & J_3^{(B,\,c)}\;:=\;\left(0,\; \omega^{1\bar1},\; \omega^{12}+B\, \omega^{1\bar2}+c\, \omega^{2\bar1} \right) \;, & c\neq\left|B-1\right|,\, \left(c,\, \left|B\right|\right)\neq\left(0,\,1\right),\, \mathcal{S}(B,c) > 0\\[3pt]
 \midrule[0.02em]
22 & \mathfrak{h}_{16}\;:=\;\left(0,\; 0,\; 0,\; 12,\; 14,\; 24\right) & 3 & 5 & 6 & J^B\;:=\;\left(0,\; \omega^{1\bar1},\; \omega^{12}+B\, \omega^{1\bar2}\right) \;, & \left|B\right|=1,\, B\neq 1 \\[3pt]
 \midrule[0.02em]
23 & \mathfrak{h}_{19}^-\;:=\;\left(0,\; 0,\; 0,\; 12,\; 23,\; 14-35\right) & 3 & 5 & 6 & J_1\;:=\;\left(0,\; \omega^{13}+\omega^{1\bar3},\; \im\, \left( \omega^{1\bar2}-\omega^{2\bar1}\right)\right) & \\[3pt]
24 & & & & & J_2\;:=\;\left(0,\; \omega^{13}+\omega^{1\bar3},\; -\im\, \left( \omega^{1\bar2}-\omega^{2\bar1}\right)\right) & \\[3pt]
 \midrule[0.02em]
25 & \mathfrak{h}_{26}^+\;:=\;\left(0,\; 0,\; 12,\; 13,\; 23,\; 14+25\right) & 2 & 4 & 6 & J_1\;:=\;\left(0,\; \omega^{13}+\omega^{1\bar3},\; \im\, \omega^{1\bar1}+\im\, \left( \omega^{1\bar2}-\omega^{2\bar1}\right)\right) & \\[3pt]
26 & & & & & J_2\;:=\;\left(0,\; \omega^{13}+\omega^{1\bar3},\; \im\, \omega^{1\bar1}-\im\, \left( \omega^{1\bar2}-\omega^{2\bar1}\right)\right) & \\[3pt]
\bottomrule
\end{tabular}
}}
\caption{M. Ceballos, A. Otal, L. Ugarte, and R. Villacampa's classification of linear integrable complex structures on $6$-dimensional nilpotent Lie algebras up to equivalence, \cite{ceballos-otal-ugarte-villacampa}.}
\label{table:classification}
\end{table}

\begin{landscape}
\begin{table}\centering
{\resizebox{1.3\textwidth}{!}{
\begin{tabular}{>{\textbf\bgroup}l<{\textbf\egroup} >{$}c<{$} >{$}c<{$} | >{$}c<{$} || c || >{$}c<{$} >{$}c<{$} | >{$}c<{$} >{$}c<{$} >{$}c<{$} | >{$}c<{$} >{$}c<{$} >{$}c<{$} >{$}c<{$} | >{$}c<{$} >{$}c<{$} >{$}c<{$} | >{$}c<{$} >{$}c<{$} || >{$}c<{$} >{$}c<{$} >{$}c<{$} || >{$}c<{$} >{$}c<{$} >{$}c<{$}}
\toprule
 & & & \lambda \geq 0 ,\; c \geq 0 ,\; B \in \C ,\; D \in \C & \textsc{skt} & h^{1,0}_{BC} & h^{0,1}_{BC} & h^{2,0}_{BC} & h^{1,1}_{BC} & h^{0,2}_{BC} & h^{3,0}_{BC} & h^{2,1}_{BC} & h^{1,2}_{BC} & h^{0,3}_{BC} & h^{3,1}_{BC} & h^{2,2}_{BC} & h^{1,3}_{BC} & h^{3,2}_{BC} & h^{2,3}_{BC} & b_1 & b_2 & b_3 & \Delta^1 & \Delta^2 & \Delta^3 \\[1pt]
 & & & \left( \mathcal{S}(B,c) := c^4-2\,\left(\left|B\right|^2+1\right)\, c^2 + \left(\left|B\right|^2-1\right)^2 \right) & & & & & & & & & & & & & & & & & & & & & \\
\midrule[0.02em]\addlinespace[1.5pt]\midrule[0.02em]
00  & \mathfrak{h}_{ 1}   & J & & \checkmark & 3 & 3 & 3 & 9 & 3 & 1 & 9 & 9 & 1 & 3 & 9 & 3 & 3 & 3 & 6 & 15 & 20 & 0 & 0 & 0 \\
\hline
01a  & \mathfrak{h}_{ 2}   & J_1^D & D \not \in \{ \im, - \im \}, \Im D \in \{ 1, -1 \}  & $\times$ & 2 & 2 & 1 & 4 & 1 & 1 & 6 & 6 & 1 & 3 & 6 & 3 & 3 & 3 & 4 & 8 & 10 & 2 & 2 & 8 \\[1pt]
01b  & \mathfrak{h}_{ 2}   & J_1^D & D \in \{ \im, -\im \} & \checkmark & 2 & 2 & 1 & 4 & 1 & 1 & 6 & 6 & 1 & 3 & 7 & 3 & 3 & 3 & 4 & 8 & 10 & 2 & 3 & 8 \\[1pt]
02a & \mathfrak{h}_{ 2}   & J_2^D & \Re D \neq 1 ,\; \left|D\right|^2 + 2 \, \Re D \neq 0 ,\; \Im D > 0 & $\times$ & 2 & 2 & 1 & 4 & 1 & 1 & 6 & 6 & 1 & 2 & 6 & 2 & 3 & 3 & 4 & 8 & 10 & 2 & 0 & 8 \\[1pt]
02b & \mathfrak{h}_{ 2}   & J_2^D & \Re D = 1 ,\; \Im D > 0 & \checkmark & 2 & 2 & 1 & 4 & 1 & 1 & 6 & 6 & 1 & 2 & 7 & 2 & 3 & 3 & 4 & 8 & 10 & 2 & 1 & 8 \\[1pt]
02c & \mathfrak{h}_{ 2}   & J_2^D & \Re D \neq 1 ,\; \left|D\right|^2 + 2\, \Re D = 0 ,\; \Im D > 0 & $\times$ & 2 & 2 & 1 & 5 & 1 & 1 & 6 & 6 & 1 & 2 & 6 & 2 & 3 & 3 & 4 & 8 & 10 & 2 & 1 & 8 \\
\hline
03  & \mathfrak{h}_{ 3}   & J_1 & & $\times$ & 2 & 2 & 1 & 4 & 1 & 1 & 6 & 6 & 1 & 3 & 7 & 3 & 3 & 3 & 5 & 9 & 10 & 0 & 1 & 8 \\[1pt]
04  & \mathfrak{h}_{ 3}   & J_2 & & $\times$ & 2 & 2 & 1 & 4 & 1 & 1 & 6 & 6 & 1 & 3 & 7 & 3 & 3 & 3 & 5 & 9 & 10 & 0 & 1 & 8 \\
\hline
05  & \mathfrak{h}_{ 4}   & J_1 & & $\times$ & 2 & 2 & 1 & 4 & 1 & 1 & 6 & 6 & 1 & 3 & 6 & 3 & 3 & 3 & 4 & 8 & 10 & 2 & 2 & 8 \\[1pt]
06a & \mathfrak{h}_{ 4}   & J_2^D & D \in \R\setminus \left\{-2,\, 0,\, 1\right\} & $\times$ & 2 & 2 & 1 & 4 & 1 & 1 & 6 & 6 & 1 & 2 & 6 & 2 & 3 & 3 & 4 & 8 & 10 & 2 & 0 & 8 \\[1pt]
06b & \mathfrak{h}_{ 4}   & J_2^D & D=1 & \checkmark & 2 & 2 & 1 & 4 & 1 & 1 & 6 & 6 & 1 & 2 & 7 & 2 & 3 & 3 & 4 & 8 & 10 & 2 & 1 & 8 \\[1pt]
06c & \mathfrak{h}_{ 4}   & J_2^D & D=-2 & $\times$ & 2 & 2 & 1 & 5 & 1 & 1 & 6 & 6 & 1 & 2 & 6 & 2 & 3 & 3 & 4 & 8 & 10 & 2 & 1 & 8 \\
\hline
07a  & \mathfrak{h}_{ 5}   & J_1^D & D \in \left(0,\, \frac{1}{4}\right) & $\times$ & 2 & 2 & 1 & 4 & 1 & 1 & 6 & 6 & 1 & 3 & 6 & 3 & 3 & 3 & 4 & 8 & 10 & 2 & 2 & 8 \\[1pt]
07b  & \mathfrak{h}_{ 5}   & J_1^D & D=0 & $\times$ & 2 & 2 & 2 & 6 & 2 & 1 & 6 & 6 & 1 & 3 & 6 & 3 & 3 & 3 & 4 & 8 & 10 & 2 & 6 & 8 \\[1pt]
08  & \mathfrak{h}_{ 5}   & J_2 & & $\times$ & 2 & 2 & 3 & 4 & 3 & 1 & 6 & 6 & 1 & 2 & 8 & 2 & 3 & 3 & 4 & 8 & 10 & 2 & 6 & 8 \\[1pt]
09a  & \mathfrak{h}_{ 5}   & J_3^{\left(\lambda,\, D\right)} & \left\{ \lambda=0 ,\; \Re D \neq\frac{1}{2} ,\; 0 < 4\, \left(\Im D\right)^2 < 1 + 4 \,\Re D ,\; \Im D > 0 \right\} & $\times$ & 2 & 2 & 1 & 4 & 1 & 1 & 6 & 6 & 1 & 2 & 6 & 2 & 3 & 3 & 4 & 8 & 10 & 2 & 0 & 8 \\[1pt]
  & & & \cup \; \left\{ 0<\lambda^2<\frac{1}{2} ,\; \Re D = 0 ,\; 0 < \Im D < \frac{\lambda^2}{2} \right\} & & & & & & & & & & & & & & & & & & & & & \\[1pt]
  & & & \cup \; \left\{ \frac{1}{2}\leq\lambda^2<1 ,\; \Re D = 0 ,\; 0 < \Im D < \frac{1-\lambda^2}{2} \right\} & & & & & & & & & & & & & & & & & & & & & \\[1pt]
  & & & \cup \; \left\{ \lambda^2>1 ,\; \Re D = 0 ,\; 0 < \Im D < \frac{\lambda^2-1}{2} ,\; \left(\Im D\right)^2 \neq \lambda^2 - 1 \right\} & & & & & & & & & & & & & & & & & & & & & \\[1pt]
09b'  & \mathfrak{h}_{ 5}   & J_3^{\left(\lambda,\, D\right)} & \left\{ \lambda = 0 ,\; \Re D = \frac{1}{2} ,\; 0 < 4\, \left(\Im D\right)^2 < 3 ,\; \Im D > 0 \right\} & \checkmark & 2 & 2 & 1   & 4 & 1 & 1 & 6 & 6 & 1 & 2 & 7 & 2 & 3 & 3 & 4 & 8 & 10 & 2 & 1 & 8 \\[1pt]
09b''  & & & \cup \; \left\{ \lambda=0 ,\; \Re D \not\in \left\{0, \frac{1}{2}\right\} ,\; 1 + 4\, \Re D > 0 ,\; \Im D = 0 \right\} & $\times$ & & & & & & & & & & & & & & & & & & & & \\[1pt]
09c  & \mathfrak{h}_{ 5}   & J_3^{\left(\lambda,\, D\right)} & \lambda=0 ,\; D = \frac{1}{2} & \checkmark & 2 & 2 & 1 & 4 & 1 & 1 & 6 & 6 & 1 & 2 & 8 & 2 & 3 & 3 & 4 & 8 & 10 & 2 & 2 & 8 \\[1pt]
09d  & \mathfrak{h}_{ 5}   & J_3^{\left(\lambda,\, D\right)} & \lambda^2>1 ,\; \Re D = 0 ,\; 0 < \Im D < \frac{\lambda^2-1}{2} ,\; \left(\Im D\right)^2 = \lambda^2-1 & $\times$ & 2 & 2 & 1 & 5 & 1 & 1 & 6 & 6 & 1 & 2 & 6 & 2 & 3 & 3 & 4 & 8 & 10 & 2 & 1 & 8 \\[1pt]
09e  & \mathfrak{h}_{ 5}   & J_3^{\left(\lambda,\, D\right)} & \left\{ 0<\lambda^2<\frac{1}{2} ,\; D=0 \right\} & $\times$ & 2 & 2 & 2 & 4 & 2 & 1 & 6 & 6 & 1 & 2 & 6 & 2 & 3 & 3 & 4 & 8 & 10 & 2 & 2 & 8 \\[1pt]
  & & & \cup \; \left\{ \frac{1}{2}\leq\lambda^2<1 ,\; D=0 \right\} & & & & & & & & & & & & & & & & & & & & & \\[1pt]
  & & & \cup \; \left\{ \lambda^2>1 ,\; D=0 \right\} & & & & & & & & & & & & & & & & & & & & & \\[1pt]
09f  & \mathfrak{h}_{ 5}   & J_3^{\left(\lambda,\, D\right)} & \lambda=0 ,\; D=0 & $\times$ & 2 & 2 & 2 & 4 & 2 & 1 & 6 & 6 & 1 & 2 & 7 & 2 & 3 & 3 & 4 & 8 & 10 & 2 & 3 & 8 \\
\hline
10  & \mathfrak{h}_{ 6}   & J & & $\times$ & 2 & 2 & 2 & 5 & 2 & 1 & 6 & 6 & 1 & 2 & 6 & 2 & 3 & 3 & 4 & 9 & 12 & 2 & 1 & 4 \\
\hline
11  & \mathfrak{h}_{ 7}   & J & & $\times$ & 1 & 1 & 2 & 5 & 2 & 1 & 6 & 6 & 1 & 2 & 5 & 2 & 3 & 3 & 3 & 8 & 12 & 2 & 2 & 4 \\
\hline
12  & \mathfrak{h}_{ 8}   & J & & \checkmark & 2 & 2 & 2 & 6 & 2 & 1 & 7 & 7 & 1 & 3 & 8 & 3 & 3 & 3 & 5 & 11 & 14 & 0 & 2 & 4 \\
\hline
13  & \mathfrak{h}_{ 9}   & J & & $\times$ & 1 & 1 & 1 & 4 & 1 & 1 & 5 & 5 & 1 & 3 & 6 & 3 & 3 & 3 & 4 & 7 & 8 & 0 & 4 & 8 \\
\hline
14  & \mathfrak{h}_{10}   & J & & $\times$ & 1 & 1 & 1 & 4 & 1 & 1 & 5 & 5 & 1 & 2 & 5 & 2 & 3 & 3 & 3 & 6 & 8 & 2 & 3 & 8 \\
\hline
15a & \mathfrak{h}_{11}   & J^B & B \in \R\setminus\left\{0,\, \frac{1}{2},\, 1\right\} & $\times$ & 1 & 1 & 1 & 4 & 1 & 1 & 5 & 5 & 1 & 2 & 5 & 2 & 3 & 3 & 3 & 6 & 8 & 2 & 3 & 8 \\[1pt]
15b & \mathfrak{h}_{11}   & J^B & B=\frac{1}{2}    & $\times$ & 1 & 1 & 1 & 4 & 1 & 1 & 5 & 5 & 1 & 2 & 6 & 2 & 3 & 3 & 3 & 6 & 8 & 2 & 4 & 8 \\
\hline
16a  & \mathfrak{h}_{12}   & J^B & \Re B\neq\frac{1}{2} ,\; \Im B \neq 0 & $\times$ & 1 & 1 & 1 & 4 & 1 & 1 & 5 & 5 & 1 & 2 & 5 & 2 & 3 & 3 & 3 & 6 & 8 & 2 & 3 & 8 \\[1pt]
16b  & \mathfrak{h}_{12}   & J^B & \Re B=\frac{1}{2} ,\; \Im B \neq 0 & $\times$ & 1 & 1 & 1 & 4 & 1 & 1 & 5 & 5 & 1 & 2 & 6 & 2 & 3 & 3 & 3 & 6 & 8 & 2 & 4 & 8 \\
\hline
17a & \mathfrak{h}_{13}   & J^{\left(B,\,c\right)} & c \not\in \left\{ \left|B-1\right|,\, \left|B\right| \right\} ,\; \left(c,\, \left|B\right|\right) \neq \left(0,\, 1\right) ,\; \mathcal{S}(B,c) < 0 ,\; B \neq 1 & $\times$ & 1 & 1 & 1 & 4 & 1 & 1 & 5 & 5 & 1 & 2 & 5 & 2 & 3 & 3 & 3 & 5 & 6 & 2 & 5 & 12 \\[1pt]
17b & \mathfrak{h}_{13}   & J^{\left(B,\,c\right)} & c = \left|B\right| > \frac{1}{2} ,\; \Re B \neq \frac{1}{2} ,\; B \neq 1 & $\times$ & 1 & 1 & 1 & 4 & 1 & 1 & 5 & 5 & 1 & 2 & 6 & 2 & 3 & 3 & 3 & 5 & 6 & 2 & 6 & 12 \\[1pt]
17c & \mathfrak{h}_{13}   & J^{\left(B,\,c\right)} & c \not\in \left\{ 0 ,\, 1 \right\} ,\; c < 2 ,\; B = 1 & $\times$ & 1 & 1 & 1 & 5 & 1 & 1 & 5 & 5 & 1 & 2 & 5 & 2 & 3 & 3 & 3 & 5 & 6 & 2 & 6 & 12 \\[1pt]
17d & \mathfrak{h}_{13}   & J^{\left(B,\,c\right)} & c = 1 ,\; B = 1 & $\times$ & 1 & 1 & 1 & 5 & 1 & 1 & 5 & 5 & 1 & 2 & 6 & 2 & 3 & 3 & 3 & 5 & 6 & 2 & 7 & 12 \\
\hline
18a & \mathfrak{h}_{14}   & J^{\left(B,\,c\right)} & B\neq 1 \;, c \not\in \left\{0,\, \left|B\right|,\, \left|B-1\right|\right\} ,\; \mathcal{S}(B,c) = 0 & $\times$ & 1 & 1 & 1 & 4 & 1 & 1 & 5 & 5 & 1 & 2 & 5 & 2 & 3 & 3 & 3 & 5 & 6 & 2 & 5 & 12 \\[1pt]
18b & \mathfrak{h}_{14}   & J^{\left(B,\,c\right)} & c = \left|B\right| = \frac{1}{2} ,\; \Re B \neq \frac{1}{2} & $\times$ & 1 & 1 & 1 & 4 & 1 & 1 & 5 & 5 & 1 & 2 & 6 & 2 & 3 & 3 & 3 & 5 & 6 & 2 & 6 & 12 \\[1pt]
18c & \mathfrak{h}_{14}   & J^{\left(B,\,c\right)} & B=1,\, c=2 & $\times$ & 1 & 1 & 1 & 5 & 1 & 1 & 5 & 5 & 1 & 2 & 5 & 2 & 3 & 3 & 3 & 5 & 6 & 2 & 6 & 12 \\
\hline
19  & \mathfrak{h}_{15}   & J_1 & & $\times$ & 1 & 1 & 1 & 5 & 1 & 1 & 5 & 5 & 1 & 3 & 5 & 3 & 3 & 3 & 3 & 5 & 6 & 2 & 8 & 12 \\[1pt]
20a & \mathfrak{h}_{15}   & J_2^c & c \not\in \left\{ 0 ,\, 1 \right\} & $\times$ & 1 & 1 & 1 & 4 & 1 & 1 & 5 & 5 & 1 & 3 & 5 & 3 & 3 & 3 & 3 & 5 & 6 & 2 & 7 & 12 \\[1pt]
20b & \mathfrak{h}_{15}   & J_2^c & c = 0 & $\times$ & 1 & 1 & 2 & 4 & 2 & 1 & 5 & 5 & 1 & 3 & 5 & 3 & 3 & 3 & 3 & 5 & 6 & 2 & 9 & 12 \\[1pt]
21a & \mathfrak{h}_{15}   & J_3^{\left(B,\,c\right)} & c \not\in \left\{0,\, \left|B-1\right|,\, \left|B\right|\right\} ,\; B\neq 1 ,\; \mathcal{S}(B,c) > 0 & $\times$ & 1 & 1 & 1 & 4 & 1 & 1 & 5 & 5 & 1 & 2 & 5 & 2 & 3 & 3 & 3 & 5 & 6 & 2 & 5 & 12 \\[1pt]
21b & \mathfrak{h}_{15}   & J_3^{\left(B,\,c\right)} & 0 < c = \left|B\right| < \frac{1}{2} & $\times$ & 1 & 1 & 1 & 4 & 1 & 1 & 5 & 5 & 1 & 2 & 6 & 2 & 3 & 3 & 3 & 5 & 6 & 2 & 6 & 12 \\[1pt]
21c & \mathfrak{h}_{15}   & J_3^{\left(B,\,c\right)} & c > 2 ,\; B=1 & $\times$ & 1 & 1 & 1 & 5 & 1 & 1 & 5 & 5 & 1 & 2 & 5 & 2 & 3 & 3 & 3 & 5 & 6 & 2 & 6 & 12 \\[1pt]
21d & \mathfrak{h}_{15}   & J_3^{\left(B,\,c\right)} & c=0 ,\; B \neq 0 ,\; \left|B\right| \neq 1 & $\times$ & 1 & 1 & 2 & 4 & 2 & 1 & 5 & 5 & 1 & 2 & 5 & 2 & 3 & 3 & 3 & 5 & 6 & 2 & 7 & 12 \\[1pt]
21e & \mathfrak{h}_{15}   & J_3^{\left(B,\,c\right)} & c=0 ,\; B = 0 & $\times$ & 1 & 1 & 2 & 4 & 2 & 1 & 5 & 5 & 1 & 2 & 7 & 2 & 3 & 3 & 3 & 5 & 6 & 2 & 9 & 12 \\
\hline
22  & \mathfrak{h}_{16}   & J^B & \left|B\right| = 1 ,\; B \neq 1& $\times$ & 1 & 1 & 2 & 4 & 2 & 1 & 5 & 5 & 1 & 2 & 5 & 2 & 3 & 3 & 3 & 5 & 6 & 2 & 7 & 12 \\
\hline
23  & \mathfrak{h}_{19}^- & J_1 & & $\times$ & 1 & 1 & 1 & 2 & 1 & 1 & 3 & 3 & 1 & 2 & 4 & 2 & 2 & 2 & 3 & 5 & 6 & 0 & 2 & 4 \\[1pt]
24  & \mathfrak{h}_{19}^- & J_2 & & $\times$ & 1 & 1 & 1 & 2 & 1 & 1 & 3 & 3 & 1 & 2 & 4 & 2 & 2 & 2 & 3 & 5 & 6 & 0 & 2 & 4 \\
\hline
25  & \mathfrak{h}_{26}^+ & J_1 & & $\times$ & 1 & 1 & 1 & 2 & 1 & 1 & 3 & 3 & 1 & 2 & 3 & 2 & 2 & 2 & 2 & 4 & 6 & 2 & 3 & 4 \\[1pt]
26  & \mathfrak{h}_{26}^+ & J_2 & & $\times$ & 1 & 1 & 1 & 2 & 1 & 1 & 3 & 3 & 1 & 2 & 3 & 2 & 2 & 2 & 2 & 4 & 6 & 2 & 3 & 4 \\
\bottomrule
\end{tabular}
}}
\caption{Dimensions of the Bott-Chern cohomology groups of the $6$-dimensional nilpotent Lie algebras endowed with a linear integrable complex structure.}
\label{table:bott-chern-6}
\end{table}
\end{landscape}

\begin{center}
\begin{table}\centering
\begin{tabular}{>{\textbf\bgroup}l<{\textbf\egroup} >{$}c<{$} | >{$}c<{$} >{$}r<{$}}
 \toprule
 $\mathbf{\sharp}$ & \text{\bfseries algebra} & \text{\bfseries complex structure} & \del\delbar\left(\sum_{j=1}^{3}\omega^j\wedge\bar\omega^j\right) \\
 \toprule
01 & \mathfrak{h}_2 & J_1^D & (D+\bar{D})\ \omega^{12\bar1\bar2} \\[3pt]
02 & & J_2^D & (-2+D+\bar{D})\ \omega^{12\bar1\bar2} \\[3pt]
 \midrule[0.02em]
05 & \mathfrak{h}_4 &  J_1 & -\frac12\ \omega^{12\bar1\bar2} \\[3pt]
06 & &  J_2^D & 2(D-1)\ \omega^{12\bar1\bar2} \\[3pt]
 \midrule[0.02em]
07 & \mathfrak{h}_5 &  J_1^D & (2D-1)\ \omega^{12\bar1\bar2} \\[3pt]
08 & & J_2 & -\omega^{12\bar1\bar2}\\[3pt]
09 & & J_3^{\left(\lambda,\, D\right)} & (-1-\lambda^2+D+\bar{D})\ \omega^{12\bar1\bar2} \\[3pt]
 \midrule[0.02em]
12 & \mathfrak{h}_8  & J & 0 \\[3pt]
\bottomrule
\end{tabular}
\caption{Non-toric $6$-dimensional nilmanifolds endowed with invariant complex structures and possibly with invariant pluriclosed metrics.}
\label{table:skt-6}
\end{table}
\end{center}

\begin{landscape}
\begin{table}\centering
{\resizebox{1.3\textwidth}{!}{
\begin{tabular}{>{\textbf\bgroup}l<{\textbf\egroup} >{$}l<{$} | >{$}l<{$} >{$}l<{$} >{$}l<{$} >{$}l<{$} | >{$}l<{$} >{$}l<{$}}
 \toprule
 $\mathbf{\sharp}$ & \text{\bfseries algebra} & b^1 & b^2 & b^3 & b^4 & \text{\bfseries complex structure} & \text{\bfseries conditions} \left( \lambda \geq 0 ,\; D \in \C \right) \\
 \toprule
00$_\text{8D}$ & \mathfrak{h}_1\times\T^2=\left(0^8\right) & 8 & 28 & 56 & 70 & J:=\left(0,\; 0,\; 0,\; 0\right) \;, & \\[3pt]
 \midrule[0.02em]
01$_\text{8D}$ & \mathfrak{h}_2\times\T^2=\left(0^4,\; 12,\; 34,\;0^2\right) & 6 & 17 & 30 & 36 & J_1^D:=\left(0,\; 0,\; \omega^{1\bar1}+D\,\omega^{2\bar2},\; 0\right) \;, & \Im D \in \{ 1, -1 \} \\[3pt]
02$_\text{8D}$ & & & & & & J_2^D:=\left(0,\; 0,\; \omega^{12}+\omega^{1\bar1}+\omega^{1\bar2}+D\,\omega^{2\bar2},\; 0\right) \;, & \Im D>0 \\[3pt]
 \midrule[0.02em]
05$_\text{8D}$ & \mathfrak{h}_4\times\T^2=\left(0^4,\; 12,\; 14+23,\;0^2\right) & 6 & 17 & 30 & 36 &  J_1:=\left(0,\; 0,\; \omega^{1\bar1}+\omega^{1\bar2}+\frac{1}{4}\,\omega^{2\bar2},\; 0\right) & \\[3pt]
06$_\text{8D}$ & & & & & &  J_2^D:=\left(0,\; 0,\; \omega^{12}+\omega^{1\bar1}+\omega^{1\bar2}+D\,\omega^{2\bar2},\; 0\right) \;, & D\in\R\setminus\{0\}\\[3pt]
 \midrule[0.02em]
07$_\text{8D}$ & \mathfrak{h}_5\times\T^2=\left(0^4,\; 13+42,\; 14+23,\;0^2\right) & 6 & 17 & 30 & 36 & J_1^D:=\left(0,\; 0,\; \omega^{1\bar1}+\omega^{1\bar2}+D\,\omega^{2\bar2},\; 0\right) \;, & D\in\left[0,\,\frac{1}{4}\right) \\[3pt]
08$_\text{8D}$ & & & & & & J_2:=\left(0,\; 0,\; \omega^{12},\; 0\right) & \\[3pt]
09$_\text{8D}$ & & & & & & J_3^{\left(\lambda,\,D\right)}:=\left(0,\; 0,\; \omega^{12}+\omega^{1\bar1}+\lambda\,\omega^{1\bar2}+D\,\omega^{2\bar2},\; 0\right) \;, & \left(\lambda,\,D\right)\in \left\{\left(0,\, x+\im y\right)\in\R\times \C \st y\geq 0\;, 4y^2<1+4x \right\} \\[3pt]
 & & & & & & & \cup \left\{\left(\lambda,\, \im y\right)\in\R\times \C \st 0<\lambda^2<\frac{1}{2} ,\; 0\leq y<\frac{\lambda^2}{2} \right\} \\[3pt]
 & & & & & & & \cup \left\{\left(\lambda,\, \im y\right)\in\R\times \C \st \frac{1}{2}\leq\lambda^2<1 ,\; 0\leq y<\frac{1-\lambda^2}{2} \right\} \\[3pt]
 & & & & & & & \cup \left\{\left(\lambda,\, \im y\right)\in\R\times \C \st \lambda^2>1 ,\; 0\leq y<\frac{\lambda^2-1}{2} \right\} \\[3pt]
 \midrule[0.02em]
12$_\text{8D}$ & \mathfrak{h}_8\times\T^2=\left(0^5,\; 12,\;0^2\right) & 7 & 22 & 41 & 50 & J\;:=\;\left(0,\; 0,\; \omega^{1\bar1},\; 0\right) & \\[3pt]
   & & & & & & & \\[3pt]
 \midrule[0.02em]
\bottomrule
\end{tabular}
}}
\caption{$8$-dimensional nilmanifolds of the type $M^6\times\T^2$ endowed with invariant complex structures and possibly with invariant pluriclosed metrics, where $M^6$ is a $6$-dimensional nilmanifold endowed with an invariant complex structure such that any invariant Hermitian metric on it is pluriclosed, and $\T^2$ is the standard $2$-dimensional torus endowed with the standard complex structure.}
\label{table:skt-8dim}
\end{table}
\end{landscape}

\begin{landscape}
\begin{table}\centering
{\resizebox{1.3\textwidth}{!}{
\begin{tabular}{>{\textbf\bgroup}l<{\textbf\egroup} >{$}c<{$} >{$}c<{$} | >{$}c<{$} || c || >{$}c<{$} | >{$}c<{$} >{$}c<{$} | >{$}c<{$} >{$}c<{$} | >{$}c<{$} >{$}c<{$} >{$}c<{$} | >{$}c<{$} >{$}c<{$} | >{$}c<{$} >{$}c<{$} | >{$}c<{$}  || >{$}c<{$} >{$}c<{$} >{$}c<{$} >{$}c<{$}}
\toprule
 & & & \lambda \geq 0 ,\; D \in \C & \textsc{skt} & h^{1,0}_{BC}  & h^{2,0}_{BC} & h^{1,1}_{BC} & h^{3,0}_{BC} & h^{2,1}_{BC} & h^{4,0}_{BC} & h^{3,1}_{BC} & h^{2,2}_{BC} & h^{4,1}_{BC} & h^{3,2}_{BC} & h^{4,2}_{BC} & h^{3,3}_{BC} & h^{4,3}_{BC}  & \Delta^1 & \Delta^2 & \Delta^3 & \Delta^4 \\
\midrule[0.02em]\addlinespace[1.5pt]\midrule[0.02em]
00$_{\text{8D}}$  & \mathfrak{h}_{ 1}   & J & & \checkmark & 4 & 6 & 16 & 4 & 24 & 1 & 16 & 36 & 4 & 24 & 6 & 16 & 4 & 0 & 0 & 0 & 0 \\
\hline
01a$_{\text{8D}}$  & \mathfrak{h}_{ 2}   & J_1^D & D \not \in \{ \im, - \im \}, \Im D \in \{ 1, -1 \} & $\times$ & 3 & 3 & 9 & 2 & 13 & 1 & 11 & 22 & 4 & 18 & 6 & 13 & 4 & 2 & 6 & 14 & 20\\[1pt]
01b$_{\text{8D}}$  & \mathfrak{h}_{ 2}   & J_1^D & D \in \{ \im, - \im \} & \checkmark & 3 & 3 & 9 & 2 & 13 & 1 & 11 & 23 & 4 & 19 & 6 & 14 & 4 & 2 & 7 & 16 & 22 \\[1pt]
02a$_{\text{8D}}$ & \mathfrak{h}_{ 2}   & J_2^D & \Re D \neq 1 ,\; \left|D\right|^2 + 2 \, \Re D \neq 0 ,\; \Im D > 0 & $\times$ & 3 & 3 & 9 & 2 & 13 & 1 & 10 & 22 & 3 & 17 & 5 & 13 & 4 & 2 & 4 & 10 & 16 \\[1pt]
02b$_{\text{8D}}$ & \mathfrak{h}_{ 2}   & J_2^D & \Re D = 1 ,\; \Im D > 0 & \checkmark & 3 & 3 & 9 & 2 & 13 & 1 & 10 & 23 & 3 & 18 & 5 & 14 & 4 & 2 & 5 & 12 & 18 \\[1pt]
02c$_{\text{8D}}$ & \mathfrak{h}_{ 2}   & J_2^D & \Re D \neq 1 ,\; \left|D\right|^2 + 2 \, \Re D =0 ,\; \Im D > 0 & $\times$ & 3 & 3 & 10 & 2 & 14 & 1 & 10 & 23 & 3 & 17 & 5 & 13 & 4 & 2 & 5 & 12 & 18 \\
\hline 
05$_{\text{8D}}$  & \mathfrak{h}_{ 4}   & J_1 & & $\times$ & 3 & 3 & 9 & 2 & 13 & 1 & 11 & 22 & 4 & 18 & 6 & 13 & 4 & 2 & 6 & 14 & 20 \\[1pt]
06a$_{\text{8D}}$ & \mathfrak{h}_{ 4}   & J_2^D & D\in\R\setminus\left\{-2,\, 0,\, 1\right\} & $\times$ & 3 & 3 & 9 & 2 & 13 & 1 & 10 & 22 & 3 & 17 & 5 & 13 & 4 & 2 & 4 & 10 & 16 \\[1pt]
06b$_{\text{8D}}$ & \mathfrak{h}_{ 4}   & J_2^D & D=1                           & \checkmark & 3 & 3 & 9 & 2 & 13 & 1 & 10 & 23 & 3 & 18 & 5 & 14 & 4 & 2 & 5 & 12 & 18 \\[1pt]
06c$_{\text{8D}}$ & \mathfrak{h}_{ 4}   & J_2^D & D=-2                          & $\times$ & 3 & 3 & 10 & 2 & 14 & 1 & 10 & 23 & 3 & 17 & 5 & 13 & 4 & 2 & 5 & 12 & 18 \\
\hline
07a$_{\text{8D}}$  & \mathfrak{h}_{ 5}   & J_1^D & D \in \left(0,\, \frac{1}{4}\right) & $\times$ & 3 & 3 & 9 & 2 & 13 & 1 & 11 & 22 & 4 & 18 & 6 & 13 & 4 & 2 & 6 & 14 & 20 \\[1pt]
07b$_{\text{8D}}$  & \mathfrak{h}_{ 5}   & J_1^D & D=0 & $\times$ & 3 & 4 & 11 & 3 & 16 & 1 & 12 & 24 & 4 & 18 & 6 & 13 & 4 & 2 & 10 & 22 & 28 \\[1pt]
08$_{\text{8D}}$  & \mathfrak{h}_{ 5}   & J_2 & & $\times$ & 3 & 5 & 9 & 4 & 15 & 1 & 12 & 24 & 3 & 19 & 5 & 15 & 4 & 2 & 10 & 22 & 28 \\[1pt]
09a$_{\text{8D}}$  & \mathfrak{h}_{ 5}   & J_3^{\left(\lambda,\, D\right)} & \left\{ \lambda=0 ,\; \Re D \neq\frac{1}{2} ,\; 0 < 4\, \left(\Im D\right)^2 < 1 + 4 \,\Re D ,\; \Im D > 0 \right\} & $\times$ & 3 & 3 & 9 & 2 & 13 & 1 & 10 & 22 & 3 & 17 & 5 & 13 & 4 & 2 & 4 & 10 & 16 \\[1pt]
 & & & \cup \; \left\{ 0<\lambda^2<\frac{1}{2} ,\; \Re D = 0 ,\; 0 < \Im D < \frac{\lambda^2}{2} \right\} & & & & & & & & & & & & & & & & & & \\[1pt]
 & & & \cup \; \left\{ \frac{1}{2}\leq\lambda^2<1 ,\; \Re D = 0 ,\; 0 < \Im D < \frac{1-\lambda^2}{2} \right\} & & & & & & & & & & & & & & & & & & \\[1pt]
 & & & \cup \; \left\{ \lambda^2>1 ,\; \Re D = 0 ,\; 0 < \Im D < \frac{\lambda^2-1}{2} ,\; \left(\Im D\right)^2 \neq \lambda^2 - 1 \right\} & & & & & & & & & & & & & & & & & & \\[1pt]
09b'$_{\text{8D}}$  & \mathfrak{h}_{ 5}   & J_3^{\left(\lambda,\, D\right)} & \left\{ \lambda = 0 ,\; \Re D = \frac{1}{2} ,\; 0 < 4\, \left(\Im D\right)^2 < 3 ,\; \Im D > 0 \right\} & \checkmark & 3 & 3 & 9 & 2 & 13 & 1 & 10 & 23 & 3 & 18 & 5 & 14 & 4 & 2 & 5 & 12 & 18 \\[1pt]
09b''$_{\text{8D}}$ & & & \cup \; \left\{ \lambda=0 ,\; \Re D \not\in \left\{0, \frac{1}{2}\right\} ,\; 1 + 4\, \Re D > 0 ,\; \Im D = 0 \right\} & $\times$ & & & & & & & & & & & & & & & & & \\[1pt]
09c$_{\text{8D}}$  & \mathfrak{h}_{ 5}   & J_3^{\left(\lambda,\, D\right)} & \lambda=0 ,\; D = \frac{1}{2} & \checkmark & 3 & 3 & 9 & 2 & 13 & 1 & 10 & 24 & 3 & 19 & 5 & 15 & 4 & 2 & 6 & 14 & 20 \\[1pt]
09d$_{\text{8D}}$  & \mathfrak{h}_{ 5}   & J_3^{\left(\lambda,\, D\right)} & \lambda^2>1 ,\; \Re D = 0 ,\; 0 < \Im D < \frac{\lambda^2-1}{2} ,\; \left(\Im D\right)^2 = \lambda^2-1 & $\times$ & 3 & 3 & 10 & 2 & 14 & 1 & 10 & 23 & 3 & 17 & 5 & 13 & 4 & 2 & 5 & 12 & 18 \\[1pt]
09e$_{\text{8D}}$  & \mathfrak{h}_{ 5}   & J_3^{\left(\lambda,\, D\right)} & \left\{ 0<\lambda^2<\frac{1}{2} ,\; D=0 \right\} & $\times$ & 3 & 4 & 9 & 3 & 14 & 1 & 11 & 22 & 3 & 17 & 5 & 13 & 4 & 2 & 6 & 14 & 20 \\[1pt]
     & & & \cup \; \left\{ \frac{1}{2}\leq\lambda^2<1 ,\; D=0 \right\} & & & & & & & & & & & & & & & & & & \\[1pt]
     & & & \cup \; \left\{ \lambda^2>1 ,\; D=0 \right\} & & & & & & & & & & & & & & & & & & \\[1pt]
09f$_{\text{8D}}$  & \mathfrak{h}_{ 5}   & J_3^{\left(\lambda,\, D\right)} & \lambda=0 ,\; D=0 & $\times$ & 3 & 4 & 9 & 3 & 14 & 1 & 11 & 23 & 3 & 18 & 5 & 14 & 4 & 2 & 7 & 16 & 22 \\
\hline
12$_{\text{8D}}$  & \mathfrak{h}_{ 8}   & J & & \checkmark & 3 & 4 & 11 & 3 & 17 & 1 & 13 & 28 & 4 & 21 & 6 & 15 & 4 & 0 & 2 & 8 & 12 \\
\bottomrule
\end{tabular}
}}
\caption{Dimensions of the Bott-Chern cohomology groups of the $8$-dimensional nilmanifolds of the type $M^6\times\T^2$ endowed with invariant complex structures and possibly with invariant pluriclosed metrics, where $M^6$ is a $6$-dimensional nilmanifold endowed with an invariant complex structure such that any invariant Hermitian metric on it is pluriclosed, and $\T^2$ is the standard $2$-dimensional torus endowed with the standard complex structure.}
\label{table:bc-skt-8dim}
\end{table}
  \end{landscape}

\FloatBarrier
\section*{Figures}

\begin{center}
\begin{figure}[h!]
{\resizebox{\textwidth}{!}{
\includegraphics[width=17.5cm]{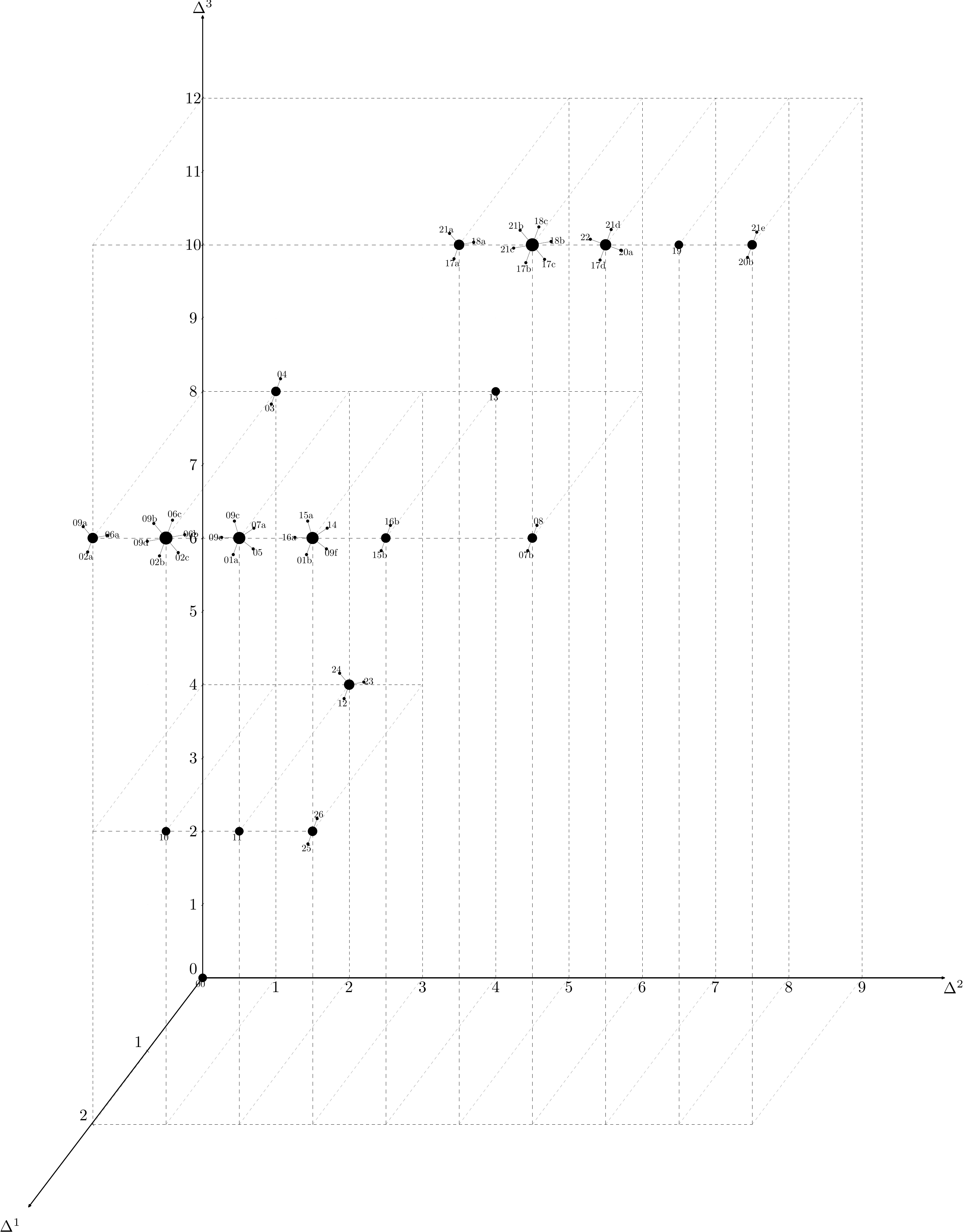}
}}
\caption{Non-K\"ahlerianity of $6$-dimensional nilmanifolds endowed with the invariant complex structures in M. Ceballos, A. Otal, L. Ugarte, and R. Villacampa's classification, \cite{ceballos-otal-ugarte-villacampa}.}
\label{fig:3d}
\end{figure}
\end{center}

\FloatBarrier

\end{document}